\documentclass[11pt,hyp,]{nyjm}
\usepackage{hyperref}
\usepackage{amsmath,amsthm,amsfonts,amssymb,amscd}
\usepackage{amsmath}
\usepackage{enumitem}
\hypersetup{nesting=true,debug=true,naturalnames=true}
\usepackage{graphicx,amssymb,upref}

\let\<\langle
\let\>\rangle

\let\uml\"

\theoremstyle{definition}
\newtheorem{defn}{Definition}[section]
\newtheorem{remark}[defn]{Remark}

\newtheorem{examp}[defn]{Example}
\newtheorem{question}[defn]{Question}

\theoremstyle{theorem}
\newtheorem{prop}[defn]{Proposition}
\newtheorem{cor}[defn]{Corollary}
\newtheorem{thm}[defn]{Theorem}
\newtheorem*{thm*}{Theorem}
\newtheorem*{lem*}{Lemma}
\newtheorem{lem}[defn]{Lemma}

\def\R{\mathbb{R}}

\def\C{\mathbb{C}}
\def\Z{\mathbb{Z}}
\def\N{\mathbb{N}}
\def\Acal{\mathcal{A}}
\def\Bcal{\mathcal{B}}
\def\wt{\widetilde}
\def\eps{\varepsilon}
\def\ra{\rightarrow}

\title{Thin Position through the lens of trisections of 4-manifolds} 

\author{Rom\'{a}n Aranda}  
\address{Department of Mathematics
14 MacLean Hall
Iowa City, Iowa 52242-1419} 
\email{jose-arandacuevas@uiowa.edu}  


\begin{document} 
\maketitle
\begin{abstract}  
Motivated by M. Scharlemann and A. Thompson's definition of thin position of 3-manifolds, we define the width of a handle decomposition a 4-manifold and introduce the notion of thin position of a compact smooth 4-manifold. We determine all manifolds having width equal to $\{1,\dots, 1\}$, and give a relation between the width of $M$ and its double $M\cup_{id_\partial} \overline M$. In particular, we describe how to obtain genus $2g+2$ and $g+2$ trisection diagrams for sphere bundles over orientable and non-orientable surfaces of genus $g$, respectively. By last, we study the problem of describing relative handlebodies as cyclic covers of 4-space branched along knotted surfaces from the width perspective. 
\end{abstract} 

\tableofcontents

\section{Introduction} 
In 1994, M. Scharlemann and A. Thompson introduced the notion of thin position of 3-manifolds. In their work \cite{thin_position_of_3m}, they described thin position as follows:

\textit{``Any closed orientable 3-manifold $M$ can be constructed as follows: begin with some 0-handles, add some 1-handles, then some 2-handles, then some more 1-handles, etc... and conclude by adding some 3-handles. Of course $M$ can be built less elaborately: in the previous description, all the 1-handles can be added at once, followed by all the 2-handles. This corresponds to a Heegaard splitting of the manifold; the 0- and 1-handles comprise one handlebody of the Heegaard splitting, the 2- and 3-handles to the other.
The idea of thin position is to build the manifold as first described, with a succession of 1-handles and 2-handles chosen to keep the boundaries of the intermediate steps as simple as possible."}

The complexity that the position of Scharlemann and Thompson seeks to minimize is the width of a handle decomposition of a 3-manifold. It is in terms of the genera of the surfaces $S$ between the 1- and 2-handles.

In dimension four, we can apply a similar reasoning to talk about thin position of a closed 4-manifold $M$ if we add the necessary 4-handles at the end of the process. In this context, the action of alternating between 1-, 2- and 3-handles becomes a suitable decomposition of a handle decomposition $M$ and the question now is what do we want ``as simple as possible" to mean. 

In 2013, D. Gay and R. Kirby \cite{trisecting_four_mans} showed that every closed smooth 4-manifold admits a trisection. 
A trisection of a closed 4-manifold $M$ is a decomposition of $M$ into three 4-dimensional 1-handlebodies with pairwise intersection being a connected 3-dimensional handlebody and triple intersection a connected closed surface.
In \cite{trisecting_four_mans} and \cite{class_trisections}, a correspondance between trisections and certain handle decompositions was described. 
In such decompositions, all the 1-handles are added at once, followed by all the 2-handles and all the 3-handles. The 0- and 1-handles comprise the first 1-handlebody of the trisection, the 3- and 4-handles form the second 1-handlebody, and the third 1-handlebody is given by a suitable neighborhood of the 2-handles. With this in mind, one can think of trisections as the 4-dimensional analogue of Heegaard splittings of $M$ with the ``trisection surface" being the triple intersection of the 4-dimensional 1-handlebodies.
 

In this paper, we use ideas of trisections of 4-manifolds from \cite{class_trisections} to define the width of a handle decomposition of a connected 4-manifold. 
This allows us to define the width of a 4-manifold $M$ by looking at the minimum width among all possible handle decompositions of $M$. Similar to \cite{thin_position_of_3m}, the width of $M$ is a multiset $\{c_i\}$ where $c_i$ is a function of some ``trisection-like surfaces" for the attaching link of the $i^{th}$ 2-handles of $M$ (see Section \ref{def_width} for the detailed definition).
 
In Section \ref{section_width}, we describe situations when the width is not minimal. The interested reader should compare these situations with the rules to decrease the width introduced in \cite{thin_position_of_3m}. We also determine all 4-manifolds satisfying $width(M)=\{1,\dots,1\}$. We show 

\newtheorem*{thm:width_only_ones}{Theorem \ref{width_1111}}
\begin{thm:width_only_ones}
Let $M$ be a connected 4-manifold satisfying $width(M)=\{1,\dots, 1\}$. Then $M$ is diffeomorphic to a (boundary) connected sum of copies of $S^1\times S^3$, $S^1\times B^3$, and linear plumbings of disk bundles over the sphere. 
\end{thm:width_only_ones} 


In Section \ref{section_rel_handlebodies}, we define the notion of the nerve of a 4-manifold $X$ obtained by adding 2-handles to $Y^3\times [0,1]$ along $Y\times \{1\}$, basically the same as the nerve of a trisection for a closed 4-manifold. 
In Subsection \ref{rel_handlebodies_width}, we show how the nerves of the suitable relative handlebodies around the 2-handles carry the width information of the handle decomposition.
We use this to show that the operation of turning a handle decomposition up-side-down leaves the width invariant. We also prove

\newtheorem*{prop:union}{Proposition \ref{width_union}}
\begin{prop:union} 
Let $M$ and $N$ be connected 4-manifolds with non-empty boundary. Suppose $f:\partial M\ra \partial N$ is a diffeomorphism between their boundaries. Then 
\[width(M\cup_f N) \leq width(M)\cup width(N)\]
\end{prop:union}

In Subsection \ref{rel_handlebodies_trisections}, we explain how handle decompositions of closed 4-manifolds of width equal to a single element correspond to trisections of 4-manifolds. This allows us to specialize some of our results to the theory of trisections of 4-manifolds (see Example \ref{sphere_bundles} and Section \ref{app_trisections}).

In Example \ref{sphere_bundles}, we give an upper bound for the width of sphere bundles over closed surfaces. In particular, we are able to obtain low genus trisection diagrams for such manifolds. 
Denote by $X_{g,n}$, $Y_{g,n}$ the disk bundles of euler number $n$ over an orientable and non-orientable surface of genus $g$, respectively. We describe how to obtain trisection diagrams of genus $2g+2$ and $g+2$ for $D(X_{g,n})$ and $D(Y_{g,n})$, respectively (see Example \ref{sphere_bundles}). It is important to mention that trisection diagrams for such manifolds have been described before in \cite{trisecting_four_mans} and \cite{Nick_Castro_Bundles}. But the existance of lower genus diagrams was proven in \cite{Simplifying_fibrations} with no explicit drawings. In Figure \ref{disk_bundles_2}, we draw explicit lower genus diagrams for $D(X_{1,0})$ and $D(Y_{1,1})$.

In Section \ref{section_symmetries}, we study the problem of describing relative handlebodies as cyclic covers of 4-space branched along knotted surfaces from the width perspective. 
We use ideas from \cite{bridge_trisection_S4}, \cite{bridge_trisections_4M} and \cite{characterizing_dehn_surgery} to prove

\newtheorem*{cor:cor_sym_trisections}{Theorem \ref{cor_sym_trisections}}
\begin{cor:cor_sym_trisections} \label{cor_sym_trisections}
Let $M$ be a closed 4-manifold. Then $M$ is the $p$-fold cyclic cover of $S^4$ branched along a knotted surface $K^2\subset S^4$ if and only if $M$ admits a $p$-symmetric trisection diagram. 
\end{cor:cor_sym_trisections} 
In \cite{thin_position_of_3m}, M. Scharlemann and A. Thompson proved that 3-manifolds of $width<\{5\}$ are 2-fold branched covers of connected sums of $S^1\times S^2$. Subsection \ref{extending_involutions} is a digression on an attempt of lifting this result to 4-manifolds with connected boundary. 

\textbf{Acknowledgements.} The author would like to thank the Topology group of the University of Iowa for listening and commenting on previous versions of this work. Special thanks to Maggy Tomova, Charles Frohman and Mitchell Messmore.


\section{Preliminaries} 

Along this work, all manifolds will be compact, smooth, oriented and connected unless the opposite is stated. For $A\subset B$ an embedded submanifold of any dimension, $\eta(A)$ will denote the closed tubular neigborhood of $A$ in $B$. 

For $0\leq k\leq n$, an $n$-dimensional $k$-handle is a copy of $D^k\times D^{n-k}$, attached to the boundary of an $n$-manifold $X$ along $(\partial D^k)\times D^{n-k}$ by an embedding $\varphi:(\partial D^k)\times D^{n-k}\ra \partial X$. By definition, 0-handles are attached along the empty set.
In dimension 3, 1-handles are attached along pairs of disjoint disks, 2-handles along annuli, and 3-handles along 2-spheres. In dimension 4, 1-handles are attached along pairs of disjoint 3-balls, 2-handles along solid tori (hence framed knots), 3-handles along thickened 2-spheres, and 4-handles along copies of $S^3$ in $\partial X$.
 
A \textbf{handle decomposition} of $X$ (relative to $\partial_-X$) is an identification of $X$ with a manifold obtained from $\partial_-X \times [0,1]$ by attaching handles along $\partial_-X\times\{1\}$. We will usually start building $X$ with $\partial_-X=\emptyset$ thus start by adding some 0-handles. 
The action on $\partial X^n$ of a $k$-handle addition $h^k=D^k\times D^{n-k}$ is by surgery on $X$. In other words, the new boundary is given by $\overline{\partial X-\varphi\left((\partial D^k)\times D^{n-k}\right)} \cup_\varphi \left( D^k\times \partial(D^{n-k})\right)$. 
For a more detailed review of the calculus on handlebody diagrams, Chapter 1 of \cite{Akbulut_notes} or Chapters 4 and 5 of \cite{gompf_stip} are good references. 

The 4-manifold obtained by attaching a handle $h$ to the 4-manifold $X$ will be denoted by $X[h]$. If $Y=\partial X$, we will write $Y[h]$ to denote the surgered 3-manifold $\partial (X[h])$. If $L$ is a framed link in $Y^3$, $Y[L]$ will denote the surgered $3-$manifold.
For simplicity, a 1-handlebody will mean a 4-ball with some 4-dimensional 1-handles attached and the three dimensional analugue will be called just handlebody. 

A \textbf{Heegaard splitting} of a closed 3-manifold $Y$ is a decomposition of $Y$ in two handlebodies with common part a connected surface of genus $g$. The smallest such $g$ is the Heegaard genus of $Y$, $HG(Y)$. 

A \textbf{trisection} of a closed 4-manifold $M$ is a decomposition of $M$ in three 1-handlebodies $M=X_1\cup X_2\cup X_3$ so that the double intersections are connected 3-dimensional handlebodies, $X_i\cap X_j =H_{i,j}$; and the triple intersection is a closed connected surface of genus $g$, $\Sigma$. It follows from the definitions that $\partial X_i = H_{i,k}\cup H_{i,j}$ is a Heegaard splitting for the boundary of $X_i$, which is homeomorphic to the connected sum of some copies of $S^1\times S^2$. The triplet $(\Sigma; H_{1,2}, H_{2,3}, H_{3,1})$ is called the \textbf{nerve} of the trisection. 

Let $L\subset Y$ be a link in a connected closed 3-manifold. Denote by $t_Y(L)$ the smallest number of embedded arcs $t\subset Y$ with $t\cap L=\partial t$ such that $\overline{Y-\eta(L\cup t)}$ is a handlebody. $t_Y(L)$ is called the \textbf{tunnel number} of $L$ in $Y$ and $t$ is a system of tunnels for $L$. Note that the surface $\Sigma =\partial \eta(L\cup t)$ induces a Heegaard splitting of $Y=H\cup_\Sigma H'$ with $L\subset core(H)$. Define $g(Y,L)$ to be the smallest genus of a Heegaard surface for $Y$ with $L$ being a subset of the core of one of the handlebodies. We have 
\[ t_Y(L)=g(Y,L)+1\] 
The equation above will serve to us as an equivalent definition of tunnel number. If $L=\emptyset$, define $t_Y(\emptyset)=HG(Y)+1$. 
We will need the following computation of the tunnel number of split links. 
\begin{lem} \label{tunnel_of_conn_sum}
Let $K_1$, $K_2$ be links inside closed three-manifolds $X_1$, $X_2$, respectively. Let $\wt{X}=X_1\#X_2$ and $\wt{K}=K_1\cup K_2\subset \wt{X}$. Then 
$$t_{\wt{X}}(\wt{K}) = t_{X_1}(K_1) + t_{X_2}(K_2) +1$$
Furthermore, if $K_1\neq\emptyset$ and $K_2= \emptyset$, then $$t_{\wt{X}}(\wt{K}) = t_{X_1}(K_1) + HG(X_2)$$
\end{lem}

\begin{proof} 
We will prove the equation for $K_1,K_2\neq \emptyset$, the other case is similar. 
One can see that the LHS is smaller by constructing a system of tunnels of cardinality $t_{X_1}(K_1) + t_{X_2}(K_2) +1$. 

We now prove $LHS\geq RHS$. Let $t\subset \wt{X}$ be a system of tunnels for $\wt{K}$ in $\wt{X}$ with $|t|=t_{\wt{X}}(\wt{K})$, and let $H=\eta (\wt{K}\cup t)$, $H'=\wt{X}-int(H)$, $\Sigma=H\cap H'$.  
By construction, $\Sigma$ is a Heegaard splitting of genus $t_{\wt{X}}(\wt{K})+1$ for $\wt{X}-int(\eta(\wt{K}))$ with $H'$ a handlebody and $H-int(\eta(K))$ a compression body with inner boundary the collection of tori given by $\partial \eta(\wt K)$. An application of Haken's Lemma gives us the existance of a sphere $S$ intersecting $\Sigma$ in one simple closed loop, separating $\wt X$ in $\big(punc(X_1),K_1\big)$ and $\big(punc(X_2),K_2\big)$, here $punc(A)$ denotes $A$ minus an open 3-ball. 
In other words, $(\Sigma; H,H')$ is the connected sum of Heegaard splittings for $X_1$ and $X_2$, say $(\Sigma_i; H_i, H'_i)$ for $i=1,2$, satisfying that $H_i -int(\eta(K_i))$ is a compression body with inner boundary the tori $\partial \eta(K_i)$. 
But recall that a compression body deformation retracts to the wedge of its inner boundary with a finite collection of arcs in the interior of the compression body with endpoints on the inner boundary. Thus $H_i$ is the tubular neighborhood of the union of $K_i$ with a collection of $t_i$ arcs with endpoints on $K_i$. This shows that $t_i \geq t_{X_i}(K_i)$. By last, notice that $t_i+1=g(\Sigma_i)$ and, since $\Sigma=\Sigma_1\# \Sigma_2$, we get 
\[ t_{\wt{X}}(\wt K) + 1 = g(\Sigma) = g(\Sigma_1) + g(\Sigma_2) = t_1+t_2 +2\geq  t_{X_1}(K_1)+  t_{X_2}(K_2) +2\]
Hence, $LHS\geq RHS$.
\end{proof}

\section{The width of a Kirby diagram}\label{section_width}
In the following section, we define the notion of the width of a handle decomposition of a 4-manifold and the width of a 4-manifold. Following \cite{thin_position_of_3m} closely, we describe specific cases when the width of a decomposition is not minimal and classify 4-manifolds with low width.

\subsection{Definition of width} \label{def_width}
Let $M$ be a connected, compact, smooth 4-manifold. Consider a handle decomposition of $M$ \[\mathcal{H}=b_0\cup C_1\cup D_1\cup E_1\cup C_2\cup D_2\cup E_2\cup \dots C_N\cup D_N\cup E_N\cup b_4\] where $b_0$, $\cup C_i$, $\cup D_i$, $\cup E_i$, $b_4$ are collections of 0-handles, 1-handles, 2-handles, 3-handles and 4-handles, respectively. We are thinking of building $M$ in steps; starting with $b_0$, then adding $C_1$, then $D_1$, then $E_1$, etc.

For each $1\leq i \leq N$, denote by $Y_i=\partial \left( b_0\cup C_1\cup D_1\cup E_1\cup \dots \cup C_i\right)$ and let $L_i \subset D_i \cap Y_i$ be the cores of the attaching regions of the 2-handles $D_i$. Note that $Y_i$ might be disconnected and that $L_i$ is a link (possibly empty) in each component of $Y_i$. For a connected component $Y'$ of $Y_i$, set $L'=L_i\cap Y'$. If $L'$ is not empty, define $c(L',Y')=2t(L',Y')+1$; if $L=\emptyset$ define $c(L',Y')=max\{2HG(Y')-1, 0\}$ where $HG(Y')$ is the Heegaard genus of $Y'$. Define $$c_i =c(L_i,Y_i):= \underset{Y'\subset Y_i}{\sum} c(L',Y').$$

\begin{defn}
The \textbf{width of the handle decomposition} $\mathcal{H}$ is the multiset $\{c_i\}$. The width of $M$ is the infimum of all widths among all possible handle decompositions of $M$. The infimum is taken with respect to the following order: 

Let $A=\{a_1 \geq a_2 \geq \dots \}$ and $B=\{b_1\geq b_2\geq \dots \}$ be two bounded multisets of $\N\cup \{0\}$ ordered in a decreasing way. Then $A< B$ if and only if there is an index $i\geq1$ such that $a_j=b_j$ for $j\leq i-1$ and $a_i<b_i$.
\end{defn}

We say that $M$ is in \textbf{thin position} if there is $M=b_0\cup C_1\cup D_1\cup E_1\cup \dots E_N\cup b_4$ with minimal $width$. It is important to mention that the infimum is always achieved; so for smooth compact $4-$manifolds the width always exists. For completeness, we include a proof of this fact in Lemma \ref{minimum_element}.

\begin{remark} 
The decomposition $b_0\cup b_4$ has empty width so $width(S^4)=\emptyset$. For non-empty diagrams, the width will be a finite multiset of odd integers $c_i\geq 1,\forall i$. 
\end{remark} 

\begin{lem} \label{minimum_element}
Let $X$ be the set of all sequences of non-negative integers with finitely many non-zero elements endowed with the order described above. Then any non-empty set has a minimal element. 
\end{lem} 

\begin{proof} 
Let $\Acal \subset X$ be non-empty. For an element $P\in X$, we define $P^{(k)}\in \N \cup \{0\}$ to be the $k-$th largest element in P. Take $\Bcal^{(0)} = \Acal$ and define the set $\Acal ^{(1)}=\{P^{(1)} | P\in \Bcal^{(0)}\}$. Since $\Acal^{(1)} \subset \N \cup \{0\}$ we can consider $\alpha _1= \min (\Acal^{(1)})$ and $\Bcal ^{(1)}=\{P\in \Bcal ^{(0)}| P^{(1)}=\alpha _1\}$. We inductively define the sets $\Bcal^{(n)}=\{P\in \Bcal ^{(n-1)}| P^{(n)}=\alpha _{n}\}$, $\Acal^{(n)}=\{P^{(n)} | P\in \Bcal^{(n-1)}\}$ and $\alpha_n= \min (\Acal^{(n)})$. By definition, the sequence $(\alpha_n)_n$ is decreasing in $\N \cup \{0\}$, so it is stationary; say $\alpha_n=\alpha\geq 0$ for all $n\geq N$. Let $P\in \Bcal^{(N)}$. By construction $P^{(j)} = \alpha_j$ for all $j\leq N$. For $l\geq 0$ we have the following inequalities, 
\[ \alpha = P^{(n)} \geq P^{(N+l)} \geq \alpha_{N+l} = \alpha \] 
Thus $P^{(m)}=\alpha$ for all $m\geq N$. But $P\in \Acal$, so all but finitely many components are non-zero. Hence $\alpha = 0$ and $P=\min (\mathcal{A})$. 
\end{proof} 

\subsection{{Ways to decrease the width}}

\begin{prop}\label{props_of_thin}
Let $\mathcal{H}=b_0\cup C_1\cup D_1\cup E_1\cup \dots \cup E_N\cup b_4$ be a thin position of $M$. For every $i$, define $Z_i=\partial b_0\left[ C_1\cup D_1\cup D_1\cup \dots \cup D_i\right]$. 
\begin{enumerate} 
\item Suppose that for some $1\leq i<N$, one of the level 3-manifolds $Z_i$, $Z_i[E_i]$, or $Z_i[E_i\cup C_{i+1}]$ is diffeomorphic to $S^3$, say X. Then $\overline{M-X}$ has no components diffeomorphic to $B^4$ unless one of them is equal to $b_0$ or $b_4$. 
\item If $D_i=E_i=\emptyset$ for some $1\leq i\leq N$, then $i=N$.
\item If $C_i=D_i=\emptyset$ for some $1\leq i\leq N$, then $i=1$. 
\end{enumerate}
\end{prop}

\begin{proof}\quad
\begin{enumerate}
\item We will prove the case $X=Z[E_i]=S^3$, the other cases are similar. Within this case, we have two options. If $M_i\approx B^4$, we can remove the handles $\{C_l\cup D_l\cup E_l | l\leq i \}$ to reduce the width. If $\overline{M-M_i}\approx B^4$, we can remove the handles $\{C_l\cup D_l\cup E_l | l> i \}$ to reduce the width. 
\item
Suppose $D_i=E_i=\emptyset$ for some $i<N$. We will show that $D_{i+1}$ and $C_{i+1}$ are both non-empty.
Suppose first $D_{i+1}=\emptyset$. In particular, $c_{i+1}=c_{i}+2|C_{i+1}|$. If we consider the new decomposition of $M$ by merging (taking the union of) $C_i$ and $C_{i+1}$ we will get the multiset $\{ c_j | j\neq i\} $, which has lower width, a contradiction. 
In a similar fashion, suppose that $C_{i+1}$ is empty: By defining $D_i$ to be $D_{i+1}$, we obtain a new decomposition of $M$ with width being the multiset $\{c_l:l\neq i\}$, which is smaller, contradicting the minimality of the width. 

We have shown that if $D_i=E_i=\emptyset$, then $D_{i+1}$ and $C_{i+1}$ are both non-empty.
Notice that $Y_{i+1}=Y_{i}[C_{i+1}]$. Thus if we take the new decomposition given by merging $C_i$ and $C_{i+1}$, then the complexity $c_{i+1}$ will not change. The width of this new decomposition will be the original width without the element $c_i\geq 1$, which is strictly lower. 

Therefore, $D_i\cup E_i$ has to be non-empty if $i<N$. 
\item 
The result follows by turning $\mathcal{H}$ up-side-down and applying Part 2 of this proposition (see Lemma \ref{lem_op}). 
\end{enumerate} 
\end{proof}

We say that a link $L\subset Y$ in a 3-manifold splits if there exists a separating sphere $S\subset Y$ disjoint from $L$. In such case, we can decompose $Y=A\#_S B$, $L=L^A\cup L^B$ with $L^A\subset A$ and $L^B\subset B$ links in the corresponding pieces, $L^A,L^B\neq \emptyset$. The following proposition studies how the width could change if one of the attaching links of the 2-handles splits in $Y_i$. 

\begin{prop} 
Let $\mathcal{H}=b_0\cup C_1\cup D_1\cup E_1\cup \dots \cup E_N\cup b_4$ be a handle decomposition for $M$. Suppose that there is a component $Y_i^*\subset Y_i$ so that the link $L_i^*=Y_i^*\cap D_i$ splits in $Y_i^*$. Write $Y_i^*=A\#_S B$, $L_i^*=L_i^A\cup L_i^B$ and let $D_i^B\subset D_i$ be the 2-handles with attaching region $L_i^B$. 

Consider the new handle decomposition $\mathcal{H}'$ of $M$ obtained from $\mathcal{H}$ by separating the 2-handles of $D_i$, adding first $D_i-D_i^B$ and then $D_i^B$. 
\begin{enumerate} 
\item If $t_B(L_i^B)>HG(B)$ or $t_A(L_i^A)>HG(A[L_i^A])$, then \[width(\mathcal{H}')<width(\mathcal{H}).\] 
\item If $t_B(L_i^B)=HG(B)$ and $t_A(L_i^A)=HG(A[L_i^A])$, then \[width(\mathcal{H}')>width(\mathcal{H}).\]
\end{enumerate}
\end{prop}

\begin{proof} 
In the new handle decomposition $\mathcal{H}'$, the $i^{th}$ level $Y_i$ got replaced by two levels: $Y^A=Z_{i-1}[E_{i-1}\cup C_i]$ and $Y^B=Y^A[D_i-D_i^B]$. Thus $width(\mathcal{H}')=width(\mathcal{H})\cup \{c^A, c^B\} -\{c_i\}$ where $c^A=c(Y^A, D_i-D_i^B)$ and $c^B=c(Y^B,D_i^B)$. 
We will show that $c^A=c_i +2\left(HG(B)-t_B(L_i^B)\right)$ and $c^B=c_i+2\left(HG(A[L_i^A])-t_A(L_i^A)\right)$. The result follows since \[t_B(L_i^B)\geq HG(B) \text{ and } t_A(L_i^A)\geq HG(A[L_i^A]).\]
By Part 1 of Lemma \ref{tunnel_of_conn_sum}, 
\[c(Y_i^*,L_i^*)=2t_{A}(L_i^A)+2t_{B}(L_i^B) +1\] 
Using the Part 2 of Lemma \ref{tunnel_of_conn_sum} and the equality $Y^A=Y_i$ we get, 
\begin{align*} 
c^A=& \underset{Y'\subset Y_i-Y_i^*}{\sum} c(Y',L') + c(Y_i^*,L_i^A)\\
=& \underset{Y'\neq Y_i^*}{\sum} c(Y',L') + 2t_{Y_i^*}(L_i^A)+1\\ 
=& \underset{Y'\neq Y_i^*}{\sum} c(Y',L') + 2t_{A\#B}(L_i^A)+1\\
=& \underset{Y'\neq Y_i^*}{\sum} c(Y',L') + 2t_{A}(L_i^A)+2HG(B)+1\\
=& \underset{Y'\neq Y_i^*}{\sum} c(Y',L') + c(Y_i^*,L_i^*)+2HG(B)-2t_{B}(L_i^B)+1\\
=& c_i +2\left(HG(B)-t_B(L_i^B)\right)
\end{align*}
Similarly, since $Y^B=Y_i-Y_i^*$ and $Y_i^*[L_i^A]\subset Y^B$, Part 2 of Lemma \ref{tunnel_of_conn_sum} gives us 
\begin{align*} 
c^A=& \underset{Y'\subset Y_i-Y_i^*[L_i^A]}{\sum} c(Y',L') + c(Y_i^*[L_i^A],L_i^B)\\
=& \underset{Y'\neq Y_i^*[L_i^A]}{\sum} c(Y',L') + 2t_{Y_i^*[L_i^A]}(L_i^B)+1\\
=& \underset{Y'\neq Y_i^*[L_i^A]}{\sum} c(Y',L') + 2t_{A[L_i^A]\#B}(L_i^B)+1\\
=& \underset{Y'\neq Y_i^*[L_i^A]}{\sum} c(Y',L') + 2t_{B}(L_i^B)+2HG(A[L_i^A])+1\\
=& c(Y_i^*,L_i^*) + 2HG(A[L_i^A])-2t_{A}(L_i^A)+1\\
\end{align*}
\end{proof} 

The following lemma states that the width of a decomposition is invariant under certain types of sliding. More precisely, if we slide handles on latter levels along previous handles, the width does not change. It follows that if a level $C_i\cup D_i$ is a $1/2-$pair of cancelling handles, we can remove such handles from the decomposition, decreasing the width. 

\begin{lem}
Let $\mathcal{H}=b_0\cup C_1\cup D_1\cup E_1\cup \dots \cup E_N\cup b_4$ be a handle decomposition of $M$. For fixed $1\leq i <j\leq N$, take connected components $A\subset L_j$ and $B\subset J_i\cup L_i$ and let $\mathcal{H}'$ be a decomposition obtained by sliding $A$ along $B$. Then $D$ and $\mathcal{H}'$ have the same width. 
\end{lem}

\begin{proof}
In terms of the attaching regions of the handles, sliding $A$ along $B$ corresponds to an isotopy of the attaching circles of $A$ in $Y_j$, which does not affect the values of $c_l$. Hence the width does not change. 
\end{proof}

\begin{lem}\label{sliding_width}
If $\mathcal{H}$ is a thin position of $M$, then $C_i\cup D_i$ does not contain a pair of cancelling 1-handle and 2-handle say, $\alpha\cup \beta$ where $\alpha$ is the attaching region of a handle of $C_1$ and $\beta\subset L_i$, so that $\alpha$ is unlinked with $L_i-\{\beta\}$ and $\beta$ goes through $\alpha$ geometrically once. \\
Similarly, $D_i\cup E_i$ does not contain a pair of cancelling 2-handle and 3-handle $\beta^2\cup \gamma^3$ where the attaching region of $\gamma^3$ is disjoint from $L_i-\{\beta\}$. 
\end{lem}

\begin{proof}
We prove the contrapositive. Suppose that there is a 1/2-cancelling pair $\alpha\cup \beta$, $\beta \in L_i$ such that $\alpha$ is disjoint from $L_i -\{\beta\}$. 
In order to erase the pair $(\alpha,\beta)$ from the diagram, we need to slide along $\beta$ the other $2-$handles of $\mathcal{H}$ intersecting $\alpha$. By assumption such $2-$handles can only be part of $L_k$ ($k>i$), so we can slide them along $\beta$ without changing $width(\mathcal{H})$ and then erase the pair $(\alpha,\beta)$, reducing the width. 
The case for 2/3-cancelling pairs is similar. It also follows from the fact that $width(\mathcal{H}^{op})=width(\mathcal{H})$ (see Lemma \ref{lem_op}), where $\mathcal{H}^{op}$ is the up-side-down handle decomposition of $\mathcal{H}$.
\end{proof}
%
\subsection{Manifolds with small width}
By taking the simplest handle decompositions of $S^1\times S^3$ and $\pm CP(2)$, we see that these manifolds have width equal to $\{1\}$. We will now show that no other $4-$manifold has such width. We also study the equation \[width(M)=\{ 1,1,\dots, 1\}.\] 
For a handle decomposition $\mathcal{H}=b_0\cup C_1\cup D_1\cup E_1\cup \dots \cup E_N\cup b_4$, denote by 
\[Y_i=\partial \left( b_0\cup C_1\cup D_1\cup E_1\cup \dots C_i\right)\]
\[Z_i=Y_i[D_i]=\partial \left(b_0\cup C_1\cup D_1\cup D_1\cup \dots \cup D_i\right)\]

\begin{prop}\label{width_1}
Let $M$ be a closed 4-manifold satisfying $width(M)=\{1\}$. Then $M$ is diffeomorphic to either $S^1\times S^3$ or $\pm CP(2)$.
\end{prop}

\begin{proof} 
Let $M=b_0\cup C_1 \cup D_1\cup E_1\cup b_4$ be a decomposition for $M$ of width equal to $\{1\}$, we can assume $|b_0|=1$. Suppose first $D_1=\emptyset$, then $C_1$ is not empty and the Heegaard genus of $Y_1=\#_{|C_1|}S^1\times S^2$ is $1$ since $c_1=1$. Thus $|C_1|=1$ and $M$ admits a handle decomposition with only one $1-$handle, i.e. $M\approx S^1\times S^3$. 

Suppose now that $D_1\neq \emptyset$. We have that $1=c_1=2t_1+1=2g_1-1$, so $t_1=0$ and $HG(Y_1)\leq g_1=1$. It follows that $|C_1|\leq 1$, giving us two options. If $|C_1|=0$, then $Y_1=S^3$ and $L_1$ is an unknot with framing say $n\in \Z$. Recall that $M$ is closed so $Y_1[L_1]$ must be a connected sum of $S^1\times S^2$. This forces $n$ to be $0$ or $\pm 1$, which implies that $M$ is either $S^4$ or $\pm CP(2)$. 

Assume the second case: $D_1\neq \emptyset$ and $|C_1|=1$. Since $t_1=0$, a neighborhood of $L_1$ induces a genus one Heegaard splitting of $S^1\times S^2$. Uniqueness of such splittings \cite{Otal_lens} implies that $L_1$ intersects $C_1$ geometrically once and that $C_1\cup D_1$ is a $1/2-$cancelling pair for $M$. In particular we conclude $M\approx S^4$.
\end{proof}


\begin{prop} \label{case11}
Let $M$ be a connected 4-manifold and let $\mathcal{H}$ be a thin position of $M$ satisfying $width(\mathcal{H})=\{1,\dots, 1\}$. Suppose that the level 3-manifolds $Y_i\left[D_i\cup E_i\right]$ are connected for all $i$. 
Then $M$ is diffeomorphic to a (boundary) connected sum of copies of $S^1\times S^3$, $S^1\times B^3$, and linear plumbings of disk bundles over the sphere. 
\end{prop}

\begin{proof} 
By definition of $c(L_i,Y_i)$, the condition $1=c_i$ implies $HG(Y_i)\leq 1$ and $HG(Y_i[D_i])\leq 1$ for all $i$. Since $Y_i[D_i]$ is a lens space and $Z_i[E_i]=Y_i[D_i\cup E_i]$ is connected, we get $|E_i|\leq 1$. Note that $|E_i|=1$ if and only if $Z_i=S^1\times S^2$ and $E_i\cap Z_i=\{pt\}\times S^2$. Thus $Z_i[E_i]=S^3$ and either $i=N$ or $M=R\# S$ where $width(R)$ and $width(S)$ are collections of at most $N-1$ ones. Thus we can assume that $E_i=\emptyset$ $\forall 1\leq i <N$. 

In a similar vein, $HG(Y_i)\leq 1$ implies $|C_i|\leq 1$, and $|C_i|=1$ if and only if $Z_i[E_i]=Z_i=S^3$, giving us a decomposition $M=R\#S$ as above. Thus we may assume that $C_i=\emptyset$ for all $1<i\leq N$. Moreover, if both $C_1$ and $D_1$ are non-empty, then $1=c_1=2t_{Y_1}(L_1)+1$ which forces $t_{Y_1}(L_1)=0$. So $L_1$ is isotopic to $S^1\times \{pt\}\subset S^1\times S^2=Y_1$. Then $C_1\cup D_1$ is a cancelling pair, contradicting the minimality of the width by Lemma \ref{sliding_width}. 
Thus by avoiding the case $M=S^1\times S^3$, we may assume that $C_i=\emptyset$ $\forall 1\leq i\leq N$. 

As before note that $t_{Y_i}(L_i)=0$; in particular, $L_i$ is an knot for all $i$. Since $L_1\subset Y_1=S^3$, $L_1$ is a unknot with framing $n_1\in \mathbb{Z}$. 
$Y_2=Y_1[L_1]=S^3[L_1]$ is a lens space $L(n_1,1)$, with $n_1\neq \pm 1$. Decompose $Y_2$ as $V_1 \cup W_1$ where $V_1=\overline{S^3-\eta(L_1)}$ and $W_1\subset L(n_1,1)$ is a solid torus so that the push-off of $L_1$ with framing $n_1$ bounds a meridian disk in $W_1$. 
Since $c_2=1$, we get $t_{Y_2}(L_2)=0$ and, by uniqueness of genus one Heegaard splittings for lens spaces \cite{Otal_lens}, we conclude that $L_2$ is isotopic (in $Y_2$) to the core of $V_1$ or $W_1$. 

Suppose $L_2$ is isotopic to the core of $W_1$. Let $f:\partial W_1 \ra \partial V_1$ be an orientation reversing homeomorphism mapping the meridian of $\partial W_1$, say $m$, to the curve with framing $n_1$ on $L_1$ in $\partial V_1$. Let $\mu, \lambda\subset \partial V_1$ be the meridian and preferred longitude of $\partial V_1$ induced by $L_1$, $\lambda$ is the boundary of the meridian of $\eta(L_1)$. By construction, $f(m)=\mu + n_1 \lambda$. We can take $f$ and $l\subset \partial W_1$ a longitude of $\partial W_1$ so that $f_* :H_1(\partial W_1) \ra H_1(\partial V_1)$ in the ordered basis $\{\lambda, \mu\}$, $\{ m,l\}$ is given by the matrix 
$\begin{pmatrix}
   n_1       & 1  \\
   1     & 0
\end{pmatrix}$.
We make such choice since the map $f$ is determined up to isotopy by $m \mapsto \mu + n_1 \lambda$. In particular, $L_2\subset W_1$ is isotopic in $W_1$ to $\lambda$ and so it can be pushed into $V_1$, becoming parallel to $\mu$. Hence we can assume that $L_2$ is isotopic to the core of $V_1$. 
Then $L_1\cup L_2$ is a Hopf link with framings say $(n_1,n_2)$.

For $i=3$, $\partial \eta(L_3) \subset Y_3 = S^3[L_1\cup L_{2}]$ is a genus one Heegaard surface for $Y_3$. 
We can decompose the lens space $Y_3$ as the union $Y_3 = W_1 \cup \overline{S^3-\eta(L_1\cup L_2)} \cup W_2$, where $W_i$ is the solid torus such that the push-off of $L_i$ with framing $n_i$ bounds a meridian disk in $W_i$ and $F_i=\partial \eta (L_i)$. $L_3\subset Y_3$ has tunnel number one so uniqueness of genus one Heegaard splittings for lens spaces \cite{Otal_lens} forces $L_3$ to be isotopic in $Y_3$ to the core of either $W_1$ or $W_2$. Without loss of generality, $L_3\subset W_2$. Using an argument analogous to that given above, we can pick an adequate attaching map $f:\partial W_2 \ra F_2 =\partial \eta L_2$ so that we can isotope $L_3$ out of $W_2$ so that $L_3$ is parallel to a meridian of $\eta(L_2)$ in $\overline{S^3-\eta(L_1\cup L_2)}$. 
Let $W_3\subset Y_4$ be the solid torus with meridian disk a curve in $\partial \eta (L_3)$ and framing $n_3$. If we sit $L_3$ in the core of $W_2$, it is clear that $W_1$ and $W_3$ induce a Heegaard splitting for $Y_4$ (note that $\overline{Y_3-(W_1\cup \eta (L_3))}$ is a product region $T^2 \times I$). 

We have shown that we can perturb $\mathcal{H}$ (without changing the width) so that $L_1\cup L_2 \cup L_3$ are isotopic to the link in Figure \ref{plumbings}. Also, notice that if $W_i$ ($i=1,2,3$) is the solid torus apearing in $Y_4$ after surgering $S^3$ along $L_1\cup L_2 \cup L_3$, then $\overline{Y_4 -(W_1\cup W_3)}$ is a product region. We can proceed inductively, repeating the argument above to conclude that $\mathcal{H}$ is the Kirby diagram of a ``linear" plumbing of disk bundles over spheres. 
\begin{figure}[h]
\centering
\includegraphics[scale=.1]{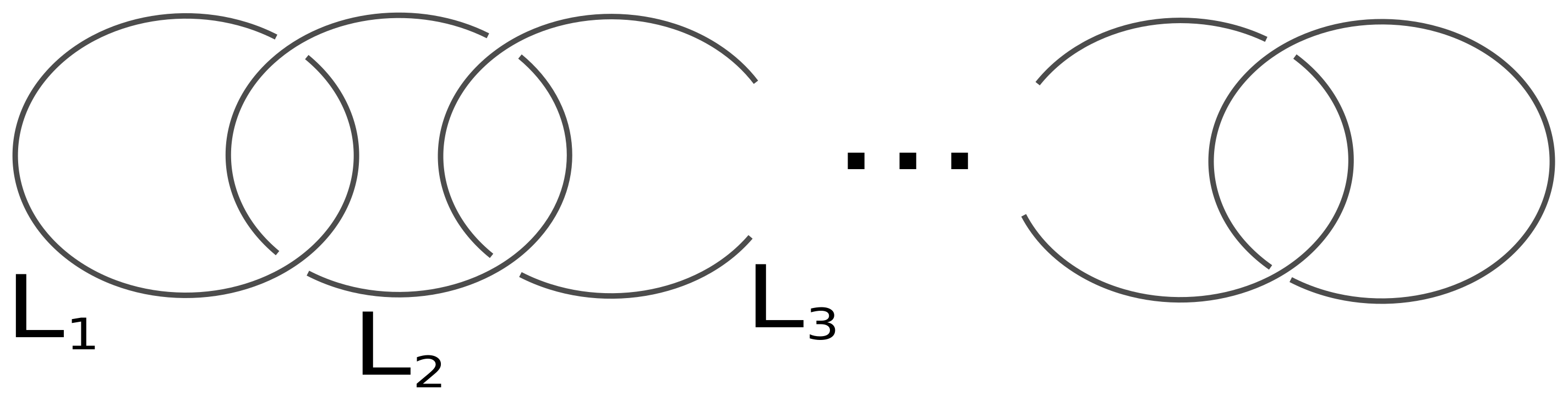}
\caption{Kirby diagram of a linear plumbing of disk bundles over a sphere.}
\label{plumbings}
\end{figure}
\end{proof} 

\begin{cor} 
For closed prime 4-manifolds, $width(M)=\{1,1\}$ if and only if $M\approx S^2\times S^2$.
\end{cor}

\begin{thm} \label{width_1111}
Let $M$ be a connected 4-manifold satisfying $width(M)=\{1,\dots, 1\}$. 
Then $M$ is diffeomorphic to a (boundary) connected sum of copies of $S^1\times S^3$, $S^1\times B^3$, and linear plumbings of disk bundles over the sphere. 
\end{thm}

\begin{proof} 
Let $\mathcal{H}=b_0\cup C_1\cup D_1\cup E_1\cup \dots \cup E_N\cup b_4$ be a thin position of $M$ of width $\{1,\dots,1\}$. The level 3-manifolds $Y_i$, $Z_i$ and $Z_i[E_i]$ may not be connected. 
By definition, $1=c_i=\underset{Y'\subset Y_i}{\sum}c(L',Y')$ where the sum runs through all connected components of $Y_i$ and $c(L',Y')\geq 0$. In particular, $HG(Y')\leq 1$ and $HG(Y'[D_i])\leq 1$ for all $Y'\subset Y_i$. Let $Y_i^*$ be the unique component of $Y_i$ such that $c(L_i^*,Y_i^*)=1$. 

Suppose that $Y_i^*=\left( \#_{j} X'_j\right)\#\left( S^1\times S^2\right)$, where $X'_j\subset Z_{i-1}[E_{i-1}]$ are some connected components. Note that $X'_j=S^3$ for all $j$. Suppose that $D_i\neq \emptyset$, then $0=t_{Y_i^*}(L_i^*)=t_{S^1\times S^2}(L_i^*)$. So there is a 1-handle $C\subset C_i$ so that $C\cup D_i$ is a 1/2-cancelling pair. Lemma \ref{sliding_width} contradicts the minimality of the width. 
We have shown that whenever $D_i\neq \emptyset$, the 1-handles of $C_i$ only act as connected sums of distinct components of $Z_{i-1}[E_{i-1}]$. In other words, $L_i$ does not intersect the attaching regions of the 1-handles of $C_i$. Moreover, since $Y_i^*=S^3$ and $t_{Y_i^*}(L_i)=0$ whenever $D_i=\emptyset$, it must be the case that $L_i$ is an unknot in $S^3$ which we can isotope in $Y_i$ (and so slide $D_i$ over previous handles) to be away from the attaching regions of $\underset{j<i}{\bigcup}\left(C_j\cup D_j\cup E_j\right) \cup C_i$ in $\partial b_0$. 

By turning $\mathcal{H}$ up-side-down, the same argument shows that if $D_i\neq \emptyset$, then the attaching sphere of each 3-handle $E'\subset E_i$ is separating in $Z_i=Y_i[D_i]$. In particular, since $HG(Y'[D_i])\leq 1$ $\forall Y'\subset Y_i$, the attaching sphere of $E'$ must bound a 3-ball in $Z_i$. We can then isotope these attaching spheres in $Z_i$ (thus slide $E_i$ over previous handles) so that $E_i\cap Y_i$ are disjoint from $\underset{j\leq i}{\cup}L_j$ in $\partial b_0$. 

We will prove the theorem by induction on the number of 2-handles of $\mathcal{H}$. If $\mathcal{H}$ has no 2-handles the result is clear. Suppose then it has at least one 2-handle. 
Pick $1\leq i_0\leq N$ to be the smallest index with $D_{i_0}\neq \emptyset$, then $Y_{i_0}^*=S^3$. $L_{i_0}^*$ is an unknot in $S^3$ so it can be isotoped to lie inside a 3-ball $G$ bounded by a 2-sphere $F\subset Y_{i_0}^*$ disjoint from the attaching region of the previous handles in $\partial b_0$. 
Suppose $D_{i_0+1}\neq \emptyset$, then $Y_{i_0+1}^*=\#_j X'_{j,i_0+1}$ for some connected components $X'_{j,i_0+1}\subset Z_{i_0}[E_{i_0}]$. By the previous paragraph, $F$ and $G$ can be picked to avoid the attaching spheres of $E_{i_0}$. We can then think of $F$ and $G$ as embedded in $Z_{i_0}[E_{i_0}]$ and so in $Z_{i_0}[E_{i_0}\cup C_{i_0+1}]$; one can do this by flowing $F$ and $G$ using the morse map induced by $\mathcal{H}$. 
If $L_{i_0+1}\subset Y_{i_0+1}^*$ lies in the same component as $F\cup G\subset Y_{i_0}=Z_{i_0}[E_{i_0}\cup C_{i_0+1}]$, then we can isotope in $Z_{i_0}[E_{i_0}]$ to lie inside $G$. To see this, recall that $Y_{i_0+1}^*=\#_j X'_{j,i_0+1}$ where all but (at most) one $X'_{j,i_0+1}$ are equal to $S^3$, and the non-$S^3$ component must contain $F\cup G$. 
Such isotopies correspond to handle slides of $D_{i_0+1}$ over previous handles which keep the width unchanged. As before, $F$ and $G$ can be picked to be disjoint from the attaching regions of $E_{i_0+1}$. 

We can continue this process of isotopying $L_{i_0+l}$ to be inside $G$ until either $D_{i_0+k}=\emptyset$ for some $k\geq 1$ or $L_{i_0+k}$ lies in a different component of $Y_{i_0+k}$  than the one containing $F\cup G$. 
In any case, this implies that the component of $Y_{i_0+k}$ containing $F\cup G$ must be $S^3$. 
Focusing our attention on the attaching region of the handles in $\partial b_0$, the sphere $F\subset \partial b_0$ will separate the attaching links $\underset{0\leq l\leq k-1}{\bigcup} L_{i_0+l}$ from the rest of the Kirby diagram for $M$. This shows that $M$ splits as a connected sum $M=R\#S$ where $R$ has a handle decomposition given by  $b_0\cup\left(\underset{0\leq l\leq k-1}{\bigcup} D_{i_0+l}\right)$. Notice here that $width(R)\leq \{1,\dots,1\}$, $width(S)\leq \{1,\dots,1\}$, $R$ satisfies the conditions of Proposition \ref{case11}, and $S$ has strictly less 2-handles than $\mathcal{H}$. This concludes the inductive step. 
\end{proof}

\section{Relative handlebodies and generalized trisections}\label{section_rel_handlebodies}
In dimension 3, thin position arguments reduce the problem to the study properties of compression bodies. In dimension 4, we will break our thinned 4-manifold in pieces that only contain the information of the 2-handles: the relative handlebodies $X(Y,L)$. In Subsection \ref{subsection_nerves}, we will see two ways of representing these relative handlebodies. Similar versions of the content of this subsection can be found in \cite{class_trisections} and \cite{trisecting_four_mans}. We include the details here due to the slight change of setting.
We will use the above in Subsection \ref{rel_handlebodies_width} to show that the width does not change when turning up-side-down a handle decomposition. 
In Subsection \ref{rel_handlebodies_trisections} we will relate the width with trisections of closed 4-manifolds. 
As an application, in Subsection \ref{subsection_double} we give an upper bound for the width of the union of two 4-manifolds, which we use to compute an upper bound for the width of sphere bundles over connected surfaces. 

\subsection{Nerves of relative handlebodies} \label{subsection_nerves}

We will describe a way to decompose 4-manifolds obtained by 2-handle attachements on collars of closed 3-manifolds. The ideas in this subsection are motivated by the notion of a Heegaard-Kirby diagrams introduced in \cite{class_trisections}. It would be interesting to see what one can say about thinned $4-$manifolds using results from dimension three and lemmas \ref{nerve_determines_everything} and \ref{lem_nerve_of_X(Y,L)}. The interested reader can look at the proofs of Proposition 3.5 and Theorem 1.2 of \cite{class_trisections} to see examples of this technique. 

For a closed surface $F$ and a collection of pairwise disjoint simple closed curves $\eps \subset F$, denote by $F|\eps$ the closed surface resulting from compressing $F$ along $\eps$. In other words, $F|\eps$ is obtained by capping-off with 2-disks the boundary components of $\overline{F-\eps}$. 

Let $\Sigma$ be a closed surface of genus $g\geq 1$ and let $\alpha, \beta, \gamma \subset \Sigma$ be three collections of $g$ pairwise disjoint non-separating simple closed curves determining three handlebodies $H$, $H'$ and $H''$, respectively. Suppose that $H\cup_\Sigma H''$ is a Heegaard splitting for $\#_k S^1\times S^2$, $k\geq 0$. 
We build a 4-manifold $Z(\Sigma;\alpha,\beta,\gamma)$ as follows:
Attach 2-handles to $\Sigma\times D^2$ along $\alpha \times \{e^{4\pi i /3}\}, \beta\times \{1\}$ and $\gamma\times\{e^{2\pi i/3}\}$ with framings induced by $\Sigma$. The resulting 4-manifold, denoted by $W_1(\Sigma;\alpha,\beta, \gamma)$, has three special 2-spheres on its connected boundary: $\Sigma|\alpha$, $\Sigma|\beta$ and $\Sigma|\gamma$.
 Attach one 3-handle along each sphere to obtain a 4-manifold, denoted by $W_2(\Sigma;\alpha,\beta, \gamma)$, with three boundary components diffeomorphic to: $H\cup_\Sigma H'$, $H'\cup_\Sigma H''$ and $H\cup_\Sigma H''$. Let $W(\Sigma;\alpha,\beta,\gamma)$ be the result of capping-off with 3- and 4-handles the boundary component of $H\cup_\Sigma H'' \approx \#_k S^1\times S^2$. 
A schematic picture of $W$ can be found on the right side of Figure \ref{nerve_of_rel_handle}.

\begin{lem} \label{nerve_determines_everything}
$W(\Sigma;\alpha,\beta,\gamma)$ is determined by the embedding of $H\cup H' \cup H''$. 
\end{lem}
\begin{proof}
The procedure of capping-off the boundary component of $W_2$ given by $H\cup_\Sigma H''$ with 3- and 4-handles is unique by \cite{uniqueness_of_34_handles}, so it is enough to show that $W_2(\Sigma;\alpha, \beta, \gamma)$ is determined by the associated handlebodies $H$, $H'$ and $H''$. Recall that two collections of $g$ pairwise disjoint non-separating simple closed curves in $\Sigma$ determine the same handlebody if and only if they differ by disk slides on the surface. Thus by the symmetry of the construction of $W_2$, it suffices to check $W_2(\Sigma; \alpha, \beta, \gamma)=W_2(\Sigma; \alpha', \beta, \gamma)$ when $\alpha$ and $\alpha'$ differ by one disk slide. In this setup, disk slides correspond to 4-dimensional 2-handle slides so $W_1(\Sigma;\alpha,\beta, \gamma)=W_1(\Sigma;\alpha',\beta, \gamma)$.

Label the components of $\alpha$ and $\alpha'$ so that $\alpha_1\neq \alpha'_1$ and $\alpha_i=\alpha'_i$ for all $i>1$. It follows that $\alpha_1$ and $\alpha'_1$ are disjoint, forcing $\overline{\Sigma-(\alpha \cup \alpha'_1)}$ to be disconnected. Let $S_1$ and $S_2$ be its two components and label them so that $S_2$ is a thrice punctured sphere with $\partial S_2=\alpha_1 \cup \alpha'_1 \cup \alpha_{j_0}$ for some $j_0>1$. Write the corresponding components of $\Sigma|(\alpha\cup \alpha'_1)$ as $\wt{S_1}\cup \wt{S_2}$.
Let $\wt W$ be the 4-manifold resulting from $W_1(\Sigma; \alpha, \beta, \gamma)$ by attaching a 2-handle along $\alpha'_1 \times \{e^{4\pi i /3}\}$ and 3-handles along $\wt{S_1},\wt{S_2}$, $\Sigma|\beta$ and $\Sigma|\gamma$.

Let $b\approx D^2\times D^2$ be the 2-handle attached along $\alpha'_1$ and $c_i$ the 3-handle attached along $\wt{S_i}$, for $i=1,2$. 
Since the framing of $b$ is given by $\Sigma$, one can check that the intersection $\Sigma\cap b=S^1\times [-1,1]$ when written in the coordinates of the handle. Also, the belt sphere of $b$, which is given by $\{0\}\times S^1$, is isotopic in $\partial b$ to a loop of the form $\{t_0\} \times S^1$ for some $t_0\in S^1$. To see that, recall that the belt sphere is a component of the Hopf link $\big(\{0\}\times S^1\big) \cup \big(S^1 \times \{0\}\big)$ in $\partial b\approx S^3$. It follows that the belt sphere of $b$ intersects $\Sigma$ in two points, one per side of $\alpha'_1$ in $\Sigma$. Thus the belt sphere intersects $S_2$ in one point. 
On the other hand, $\wt{S_2}$ is the union of $S_2$ with the cores of the 2-handles given by $\alpha_1$, $\alpha'_1$ and $\alpha_{j_0}$, which are disjoint from the belt sphere of $b$. Hence the belt sphere of $b$ will intersect $\wt{S_2}$ geometrically once. Thus the 3-handle $c_2$ cancels with $b$. Similarly, $\wt{S_1}$ intersects the belt of $b$ once. 
In order to eliminate $b$ and $c_2$, it is necessary to slide $c_1$ along $c_2$; this will change the attaching sphere of $c_1$ from $S_1$ to $\Sigma|\alpha$. Therefore, $\wt{W} \approx W_2(\Sigma; \alpha, \beta, \gamma)$. 
Analogously, one can show that $ \wt{W} \approx W_2(\Sigma; \alpha', \beta, \gamma)$, and so $W$ only depends on the handlebodies $H\cup H'\cup H''$.
\end{proof}

\begin{defn} \label{nerve}
Let $H,H',H''$ be three connected handlebodies with common boundary a surface $\Sigma$. We say that the tuple $T=(\Sigma; H,H',H'')$ is a \textbf{nerve of W} if $W\approx W(\Sigma;\alpha, \beta, \gamma)$ for some collections of curves $\alpha, \beta, \gamma \subset \Sigma$ determining $H,H',H''$, respectively. 
\end{defn}

Let $L\subset Y$ be a framed link inside a closed 3-manifold $Y$, and let $X(Y,L)$ be the smooth 4-manifold built from $Y\times [0,1]$ by attaching 2-handles along $L\times\{1\}$. 

There is a correspondance between 4-manifolds of the form $W(H,H',H'')$ and $X(Y,L)$. For the closed case, $W$ is a trisection diagram (see Subsection \ref{rel_handlebodies_trisections}), and the correspondance has been proven in \cite{trisecting_four_mans}. The proofs can be extended to our context.

The following lemma is essentially Lemma 4.1 of \cite{bridge_trisections_4M} or Lemma 14 of \cite{trisecting_four_mans}. For completeness and due to the slight change of setting (closed case vs relative case), we include a proof. 

\begin{lem} \label{lem_nerve_of_X(Y,L)}
Let $L\subset Y$ be a framed link inside a closed 3-manifold. Then there exist handllebodies $H$, $H'$, $H''$ so that $X(Y,L)\approx W(\Sigma; H,H',H'')$.
\end{lem} 
\begin{proof}
Take $(\Sigma; H, H')$ a Heegaard splitting of genus $g\geq 1$ of $Y$ so that $L\subset H$ is a subset of its core; i.e. $(\Sigma,L)$ is an admissible pair. Then $H$ after surgery along $L$ is still a 3-dimensional handlebody denoted by $H''$. 
Let $h:X \ra [0,2]$ be a Morse function arising from the construction of $X$ with $h^{-1}(0)=Y\times \{0\}$, $h^{-1}(2)\approx Y[L]$, and all critical points of index 2 at $t=1$. The flow of $h$ restricted to $H'\times \{0\}$ induces an injective isotopy in $X$ between $H'\times \{0\}$ and the handlebody $\overline{Y[L] - H''}$, thus a product region. We can use the latter region to add a copy of $\Sigma \times I$ to $H''$ and assume the three handlebodies to intersect simultaneously at the surface $\Sigma$ (see Figure \ref{nerve_of_rel_handle}).
By construction, $H\cup_\Sigma H''$ is a Heegaard splitting for a connected sum of $g-|L|$ copies of $S^1\times S^2$. 

Let $X_2$ be the 4-manifold given by flowing $H\times \{0\}$ with $h$. $X_2$ is obtained by adding 2-handles to $H\times [0,1]$ along $L\times \{1\}$. Note that $H\times [0,1]$ is a 4-dimensional 1-handlebody of genus $g$ with one 0-handle and $g$ 1-handles, where the $1-$handles are in correspondance with any set of $g$ meridians determining $H$. Since $L$ is a subset of the core of $H$, 2-handles along $L$ cancel $|L|$ 1-handles of $H\times [0,1]$. Hence $X_2$ is a 1-handlebody of genus $g-|L|$. It follows that $X(Y,L)\approx W(H,H',H'')$.
\end{proof} 

\begin{figure}[h]
\centering
\includegraphics[scale=.25]{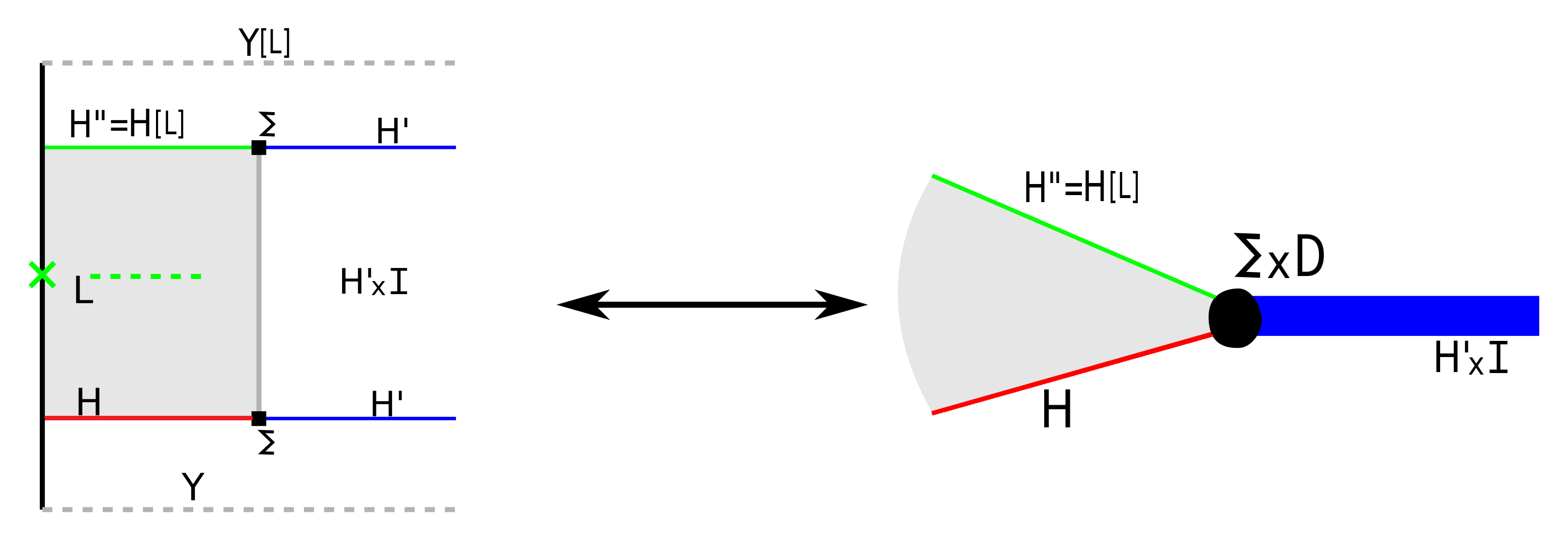}
\caption{Schematic diagram of the Morse function of $X(Y,L)$ (left) and $W(\Sigma; H,H',H'')$ (right).}
\label{nerve_of_rel_handle}
\end{figure}

\begin{lem} \label{lem_nerve_of_X(Y,L)_dual}
Let $\Sigma$ be a surface and $H,H',H''$ be handlebodies with boundary $\Sigma$ satisfying $H\cup_\Sigma H''=\#_k S^1\times S^2$ for some $0\leq k\leq g$. Then for $Y:=H\cup_\Sigma H'$, there is a link $L\subset Y$ satisfying $X(Y,L)\approx W(H,H',H'')$. 
\end{lem} 
\begin{proof} 
Since $H\cup_\Sigma H''=\#_k S^1\times S^2$, there are collections of $g$ curves $\alpha,\gamma\subset \Sigma$ determining $H$ and $H''$, respectively, such that $\alpha_l = \gamma_l$, for all $1\leq l \leq k$ and $|\alpha_i \cap \gamma_j|=\delta_{i,j}$, for all $k< i,j\leq g$. The desired link is given by $L=\{\gamma_i: k<i\leq g\}$. Notice that the $\beta$ curves, which induce $H'$, are useful to determine the embedding of $\Sigma$ into $H\cup H'$ and so the embedding of $L$. 
\end{proof} 

\begin{remark}\label{nerve_S3}
Notice that if $L=\emptyset$, then $X\approx Y\times I$, and $(\Sigma; H,H',H''=H)$ is a nerve for $X$ for any $H\cup_\Sigma H'$ Heegaard splitting of $Y$. In particular, $(S^2;B^3,B^3,B^3)$ is a nerve for $S^3\times I$.
\end{remark}

\begin{remark}\label{W_op}
We can build an up-side-down version of $X(Y,L)$ from $Y[L]$ by attaching 2-handles to $Y[L]\times [0,1]$ along $L'\times\{1\}$, where $L'\subset Y[L]$ are co-cores of the 2-handles along $L$. 
In terms of nerves, one can see that if $X(Y,L)=W(\Sigma;H,H',H'')$, then this up-side-down presentation for $X(Y,L)$ can be described by $W(\Sigma;H'',H',H)$. 
\end{remark}

\subsection{Relative handlebodies and width}\label{rel_handlebodies_width}

Let $\mathcal{H}=b_0 \cup C_1\cup D_1\cup E_1\cup \dots \cup E_N \cup b_4$ be a handle decomposition of $M$. Recall that \[Y_i=\partial \left( b_0\cup C_1\cup D_1\cup E_1\cup \dots C_i\right) \text{ and } Z_i=Y_i[D_i].\] 
For each $1\leq i\leq N$, denote by $W_i=X(Y_i,L_i)$, note that $W_i$ may be disconnected. 
Let $(\Sigma_i;H_i,H_i',H_i'')$ be nerves for $W_i$ with smallest genus\footnote{For closed disconnected surfaces, the genus is the sum of the genera of the connected components.} satisfying $W_i=X(Y_i,L_i)=W(\Sigma_i;H_i,H_i',H_i'')$.
$M$ can be decomposed as 
\[ \left[b_0\cup C_1\right] \bigcup_{Y_1} W_1 \bigcup_{Z_1} \left[\eta(Z_1)\cup E_1\cup C_2\right] \bigcup_{Y_2} \dots \bigcup_{Y_N} W_N \bigcup_{Z_N} \left[\eta(Z_N)\cup E_N\cup b_4\right].\] 
Note that the proofs of Lemmas \ref{lem_nerve_of_X(Y,L)} and \ref{lem_nerve_of_X(Y,L)_dual} show that
\[ width(\mathcal{H})=\left\{ \underset{\Sigma'\subset \Sigma_i}{\sum} max\{2g(\Sigma')-1,0\}\mid i=1,\dots,N\right\}.\]
Let $\mathcal{H}^{op}=\wt b_0 \cup \wt C_1\cup \wt D_1\cup \wt E_1\cup \dots \cup \wt E_N \cup \wt b_4$ be the up-side-down handle decomposition of $M$, where $\wt b_0 = b_4$, $\wt C_i=E_{N-i}$, $\wt D_i = D_{N-i}$, $\wt E_i= C_{N-i}$ and $\wt b_4=b_0$. The relative handlebodies $\wt W_i$ satisfy (see Remark \ref{W_op}), 
\[ \wt W_i = \left(W_{N-i}\right)^{op}=W(\Sigma_i;H_i,H_i',H_i'')^{op}=W(\Sigma_i;H_i'',H_i',H_i).\]
We have proven the following
\begin{lem}\label{lem_op}
Let $\mathcal{H}$ be a handle decomposition of $M$. Then
\[ width(\mathcal{H})=width(\mathcal{H}^{op}).\] 
\end{lem}

\subsection{Relative handlebodies and trisections of 4-manifolds}\label{rel_handlebodies_trisections}

An interesting case of the discussion in Subsection \ref{rel_handlebodies_width} is when $M$ is closed and the handle decomposition is self-indexed; i.e., $\mathcal{H}=b_0\cup C_1\cup D_1\cup E_1\cup b_4$. 
Here, the decomposition in relative handlebodies becomes 
\begin{align*}
M=& \left[b_0\cup C_1\right]\bigcup_{Y_1} W_1 \bigcup_{Z_1} \left[\eta(Z_1)\cup E_1\cup b_4\right] \\
=& \left[b_0\cup C_1\right]\bigcup_{Y_1} W(\Sigma;H,H',H'') \bigcup_{Z_1} \left[\eta(Z_1)\cup E_1\cup b_4\right] \\
=&\left[b_0\cup C_1\right]\bigcup X(H,L) \bigcup \left[E_1\cup b_4\right],
\end{align*}
where $Y_1=H\cup_\Sigma H'$, $Z_1=H''\cup_\Sigma H'$, $L\subset H$ is the attaching region of the 2-handles $D_1$ and $X(H,L)$ is the cobordism between $H$ and $H[L]$ (see Figure \ref{nerve_of_rel_handle}). 

Note that the last expression is a trisection of $M$.
Hence, $width(\mathcal{H})=\{max(2g(\Sigma)-1,0)\}$, where $\Sigma$ is the smallest genus trisection surface for $M$ inducing the handle decomposition $\mathcal{H}$. 

One would expect closed 4-manifolds with $width(M)$ being a singleton to be ``small". For example, Proposition \ref{width_1} implies that the only 4-manifolds of trisection genus one are $S^1\times S^3$ and $\pm CP2$. 

For a generic handle decomposition of $M$, we can think of a \textbf{generalized trisection} to be the collection of nerves for the relative handlebodies $\{X(Y_i,L_i)\}_i$, together with data describing the 1-handles and 3-handles pasting them along some boundary components. With this philosophy in mind, in some part of the rest of this work we will derive results about trisections of 4-manifolds as a particular case of constructions on thin position (see Example \ref{sphere_bundles} and Section \ref{app_trisections}). 

\subsection{{Width under specific operations}} \label{subsection_double}
Let $M$ be a 4-manifold with non-empty boundary and denote by $D(M)$ the double of $M$; i.e. $D(M)\approx M\cup_{id_\partial}\overline{M}$. If $\partial M$ is connected, a Kirby diagram for the double of $M$ can be built by adding an unknotted 2-handle with framming zero around each 2-handle of a Kirby diagram for $M$. 
Let $\mathcal{H}=b_0\cup C_1\cup D_1$ be a self-indexed handle decomposition in thin position for $M$ with no 3-handles. Denote by $\wt L$ the attaching links of the new 2-handles for $D(M)$. An application of the slam-dunk move \cite{gompf_stip}, shows that the components of the link $\wt L$ are isotopic in $S^3[C_1\cup D_1]$ to the cores of the solid tori after surgery along $L$. Thus the tunnel number of $\wt L$ in $S^3[C_1\cup D_1]$ is the same as the tunnel number of $L$ in $S^3[C_1]$. In particular, $width(D(M))\leq width(M)\cup \{2t_{S^3[C_1]}(L)+1\}$. 

We can refine the above argument as follows. 

\begin{prop}\label{width_union}
Let $M$ and $N$ be connected 4-manifolds with non-empty boundary. Suppose $f:\partial M\ra \partial N$ is a diffeomorphism between their boundaries. Then 
\[width(M\cup_f N) \leq width(M)\cup width(N)\]
\end{prop}

\begin{proof} 
Let $\mathcal{H}_M$ and $\mathcal{H}_N$ be a thin position for $M$ and $N$, respectively, and consider $\mathcal{H}_M\cup \mathcal{H}_N^{op}$ a decomposition of $M\cup_f N$ given by doing $\mathcal{H}_M$ first, followed by $\mathcal{H}_N$ in the opposite order. By construction,
\[ width(M\cup_f N)\leq width(\mathcal{H}_M) \cup width(\mathcal{H}_N^{op}).\] 
The result follows from Lemma \ref{lem_op}.
\end{proof} 

\begin{cor} \label{double}
Let $M$ be a 4-manifold with non-empty boundary. Then 
\[width(D(M)) \leq width(M)\cup width(M)\] 
\end{cor}

\begin{remark}
It is important to mention that if one decides to add the 2-handles $\wt L_i$ after $L_i$ (or at the same time), the tunnel number of $L_{i+1}$ in $Y_{i+1}$ will be equal to the tunnel number of $L_{i+1}$ in $Z_{i}[C_{i+1}]$; allowing the complexity $c_{i+1}$ to change with no control. Hence, the upper bound on Corollary \ref{double} is expected to be sharp only for special cases. 
\end{remark}

\begin{examp}[Sphere bundles over surfaces] \label{sphere_bundles} Let $g,n\in \Z$, $g>0$ and denote by $S_g$, $N_g$ the orientable and non-orientable surface of genus $g$, respectively. Let $X_{g,n},Y_{g,n}$ be the disk bundles over $S_g$ and $N_g$ with Euler number $n$, respectively. 
Kirby diagrams for $X$ and $Y$ with only one 2-handle are known (Fig. \ref{disk_bundles_1}), so stimates for the width of such 4-manifolds can be found. More explicitly, for $g>0$, $width(X_{g,n})\leq\{4g+1\}$ and $width(Y_{g,n})\leq \{ 2g+1\}$. 
Using Proposition \ref{double}, we obtain estimates for the width of the corresponding doubles; i.e. sphere bundles over surfaces. 
\[  width(D(X_{g,n}))\leq\{4g+1,4g+1\} \quad \text{ and } \quad width(D(Y_{g,n}))\leq \{ 2g+1,2g+1\} \]

For this particular examples, one can add all 2-handles of $D(X_{g,n})$ (resp. $D(Y_{g,n})$) at the same time and get tunnel numbers $2g+1$ (resp. $g+1$). Lemma \ref{lem_nerve_of_X(Y,L)} together with the discussion on Subsection \ref{rel_handlebodies_trisections} imply that, to draw a genus $m+1$ trisection surface for $M$, it is enough to find a system of $m$ tunnels for the attaching region of the 2-handles\footnote{This is Lemma 2.3 of \cite{characterizing_dehn_surgery}.}. Thus we can draw diagrams for  $D(X_{g,n})$ (resp. $D(Y_{g,n})$) of genus $2g+2$ (resp. $g+2$). For completeness, we draw the diagrams for genus 1 case in Figure \ref{disk_bundles_2}.

\begin{figure}[h]
\centering
\includegraphics[scale=.035]{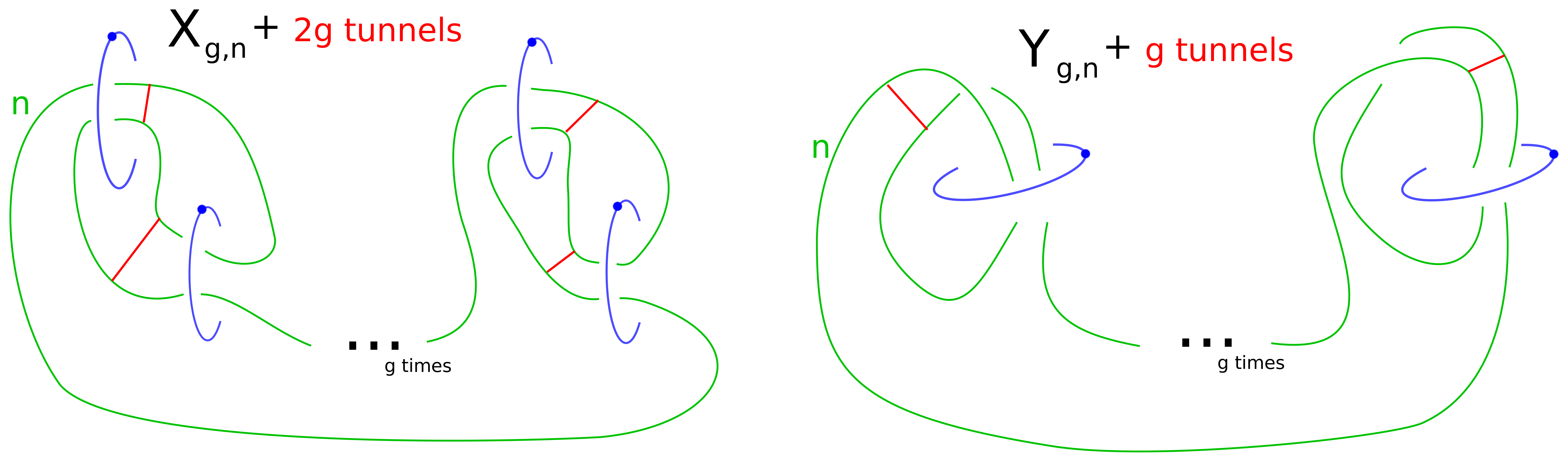}
\caption{Kirby diagrams for disk bundles over closed surfaces with specific systems of tunnels.}
\label{disk_bundles_1}
\end{figure}

\begin{figure}[h]
\centering
\includegraphics[scale=.037]{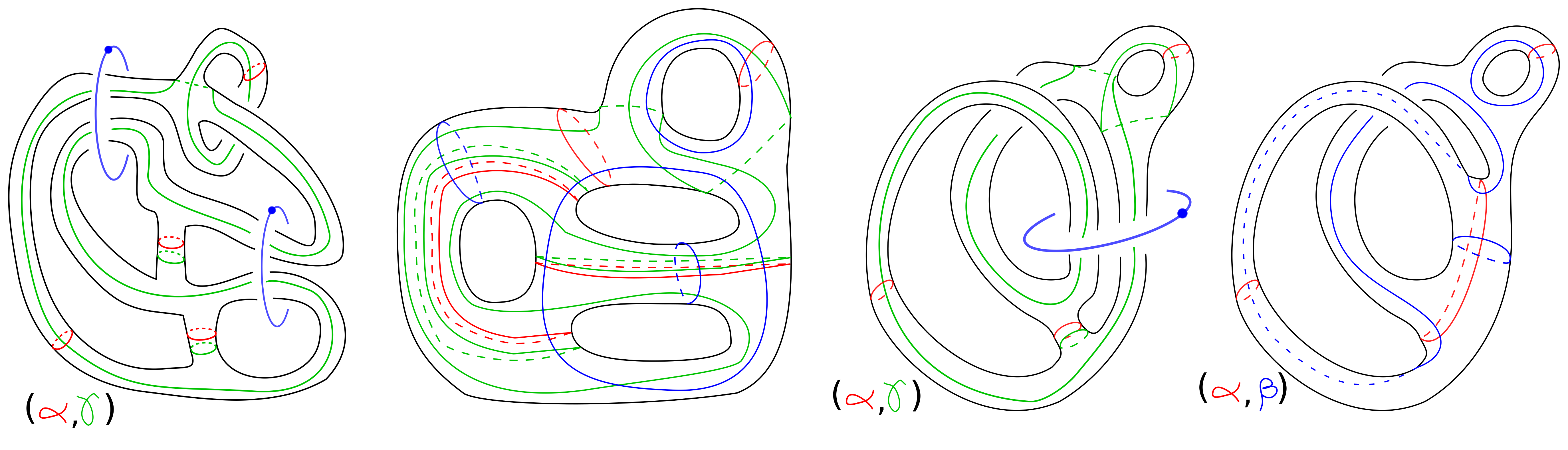}
\caption{Trisection diagrams $D(X_{1,0})\approx T^2\times S^2$ (left) and $D(Y_{1,1})\approx RP^2\wt{\times}S^2$ (right). To change the euler number it is enough to add twisting to the longer $\gamma-$curve. $\alpha$ (red), $\beta$ (blue) and $\gamma$ (green) curves correspond to curves that bound disks in $H$, $H'$ and $H''$, respectivelly.}
\label{disk_bundles_2}
\end{figure}
\end{examp}

\section{Symmetries of relative handlebodies} \label{section_symmetries}

In \cite{thin_position_of_3m}, M. Scharlemann and A. Thompson proved that 3-manifolds of $width<\{5\}$ are 2-fold branched covers of connected sums of copies of $S^1\times S^2$. 
In this section, we relate ideas of trisections of 4-manifolds from \cite{class_trisections}, \cite{bridge_trisection_S4} and \cite{characterizing_dehn_surgery} to discuss an attempt of lifting this result to 4-manifolds with connected boundary. In Subsection \ref{subsection_symmetric} we study symmetries on the relative handlebodies $X(Y,L)$. We use this in Subsection \ref{app_trisections} to talk about symmetric trisection diagrams and to prove a ``trisected version" of a Theorem of J. Birman and H. Hilden.
In Subsection \ref{extending_involutions} we study the extension problem: how to paste symmetric pieces of $M$.

\subsection{Symmetric nerves} \label{subsection_symmetric}
In this subsection $X$ will denote a 4-manifold of the form $X=X(Y,L)$ for some link $L\subset Y^3$.

\begin{defn}We say that a nerve $T=(\Sigma;H,H',H'')$ for $X$ is \textbf{p-symmetric} if there is a piecewise-linear homeomorphism $\tau:\Sigma\ra \Sigma$ of finite period $p$, extending to the interior of each handlebody, satisfying
\begin{enumerate}
\item For each handlebody, the orbit space by the action of $\tau$ is a 3-ball.
\item $Fix(\tau) =Fix(\tau^k)$ for all $1\leq k\leq p$.
\item The image of the fix set of $\tau$ on each handlebody is an unknotted set of arcs in the quotient.
\end{enumerate}
\end{defn}

\begin{remark}
The 2-symmetric condition of a nerve $T=(\Sigma;H,H',H'')$ is equivalent to the existance of an involution $\tau$ of $\Sigma$ extending to the interior of each handlebody such that $\tau$ is conjugate to the hyperelliptic involution.
\end{remark}

Using the ideas of \cite{bridge_trisection_S4}, one can show that if $T$ is a $p$-symmetric nerve of $X$, then the finite order map $\tau:\Sigma\ra \Sigma$ extends to $X$.

\begin{prop} \label{branching_nerve}
Let $L\subset Y$ be a framed link inside a closed 3-manifold and let $T=(\Sigma;H,H',H'')$ be a nerve for $X=X(Y,L)$. If $T$ is $p$-symmetric, then $X$ is a $p$-fold cyclic covering of $S^3\times I$ branched along a properly embedded knotted surface with boundary in both $S^3\times\{0,1\}$.
\end{prop}

\begin{proof}
By fixing a model handlebody of genus $g$, we can assume $H,H',H''$ are standard and take maps $f_{\alpha\beta},f_{\beta\gamma},f_{\gamma\alpha}$ between the corresponding boundaries codifying the pairwise intersections. Let $\tau$ be the order $p$ homeomorphism of $\Sigma$ extending to the three handlebodies. By assumption, $\tau$ commutes with the $f$-maps and hence the maps descend to the quotients $B:=H/\tau$, $B':=H'/\tau$, $B'':=H''/\tau$ (Figure \ref{Mod_involution}).

By definition the quotient map on each handlebody $q_\tau:H \ra B^3$ is a $p$-fold cyclic branched cover of $B^3$ along a collection of $b$ boundary parallel arcs in $B^3$, say $\theta_\alpha$, $\theta_\beta$ and $\theta_\gamma$. 

\begin{figure}[h]
\centering
\includegraphics[scale=.17]{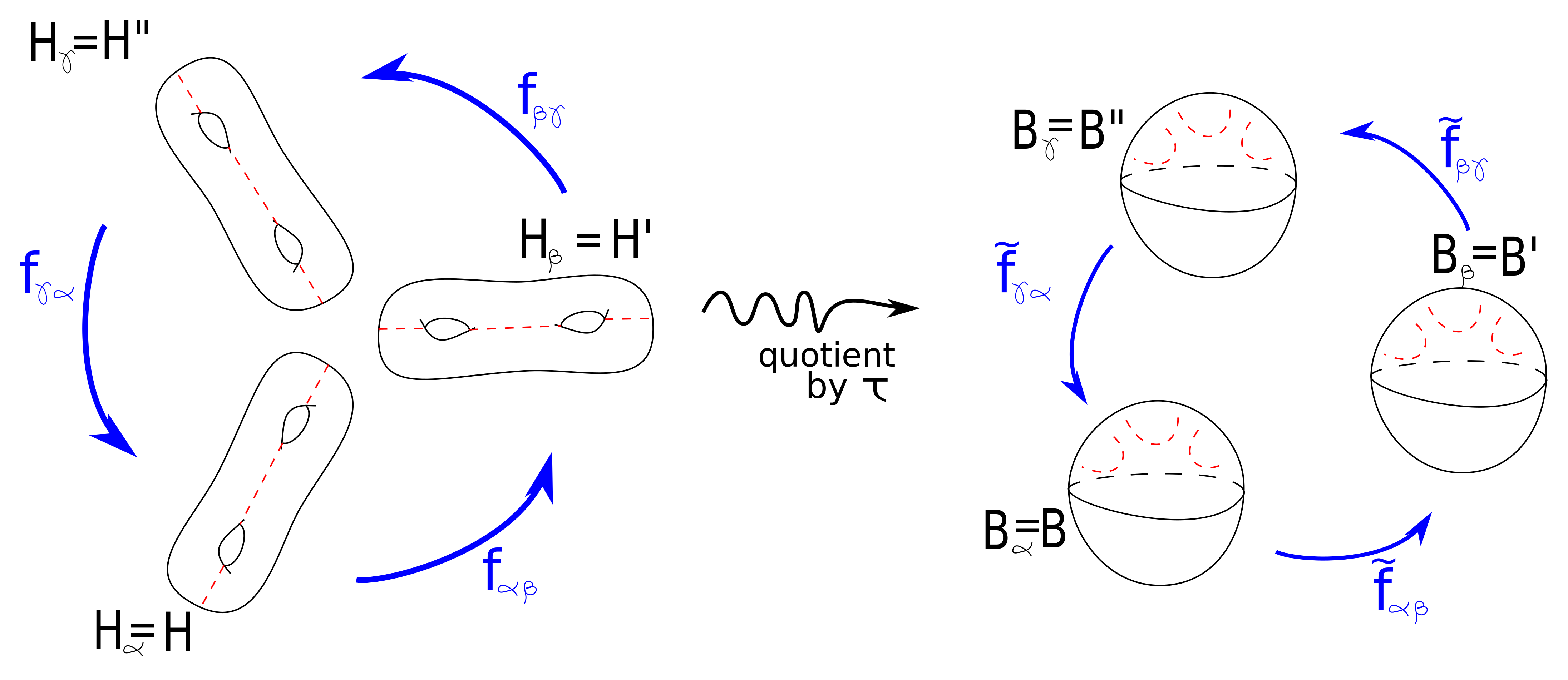}
\caption{The nerve of $T$ descends to a bridge trisection diagram on $S^3\times I$.}
\label{Mod_involution}
\end{figure}

Since, $L$ is a subset of the core of $H$, each component of $L$ is dual to a meridian disk of $H$, and the 3-manifold $H\cup_{\Sigma}H''$ is a connected sum of $k:=g-|L|$ copies of $S^1\times S^2$. 
A corollary in Section 2 of \cite{Plotnick} states that such 3-manifolds arise as $p-$fold cyclic branched cover of $S^3$ only when the branched set is an unlink. 
Thus $\theta_\gamma\cup\theta_\alpha$ is an unlink in $B\cup B'' \approx S^3$, which bounds a unique collection $\mathcal{D}$ of trivial disks in $B^4$ by Lemma 2.3 of \cite{bridge_trisections_4M}. By Remark \ref{nerve_S3}, the tuple $(\Sigma/\tau; B, B', B'')$ is a nerve for $S^3\times I$, and one can complete $\theta_\alpha \cup \theta_\beta \cup \theta_\gamma \subset B\cup B' \cup B''$ to a properly embedded surface $K^2\subset S^3\times I$ by attaching the collection $\mathcal{D}$ along $\theta_\gamma \cup \theta_\alpha$. By construction, 
\begin{align*} 
\big(K\cap (S^3\times \{0\}), S^3\times \{0\}\big) = & \big( \theta_\alpha\cup \theta_\beta, B\cup_{\Sigma/\tau} B'\big)\\
\big(K\cap (S^3\times \{1\}), S^3\times \{1\}\big) = & \big( \theta_\gamma\cup \theta_\beta, B''\cup_{\Sigma/\tau} B'\big)
\end{align*}

We can now take the $p$-fold cyclic covering of $S^3\times I$ branched along $K^2$ and lift the nerve of $S^3\times I$ to a nerve for the resulting 4-manifold. Recall that the $p$-fold cyclic cover of a 4-ball branched along a collection of trivial disks is also 4-dimensional 1-handlebody. By construction, the new tuple is indeed equal to the original nerve for $X$ $(\Sigma; H, H', H'')$. Lemma \ref{nerve_determines_everything} concludes that $X$ is the $p$-fold cyclic branched covering of $S^3\times I$ along $K$ and that the map $\tau$ extends to all $X$. 
\end{proof} 

\begin{remark} 
Suppose that a $p$-symmetric nerve $T=(\Sigma; H, H', H'')$ comes from a Kirby diagram $\mathcal{H}=b_0\cup C_1\cup E_1$ with 1-handles $E_1$ and 2-handles $D_1$, that is, $Y=\#_{|C_1|}S^1\times S^2$ and $\Sigma=\partial \eta(L\cup t)$ where $t$ is a system of tunnels for the attaching link of the 2-handles $L$ in $Y$. Since $H\cup_\Sigma H'=Y=\#_{|C_1|}S^1\times S^2$, we can use the argument in Proposition \ref{branching_nerve} to obtain a $\mathbb{Z}_p$ action on $M:=b_0[C_1\cup D_1]$ which provides a description of $M$ as a $p$-fold cyclic cover of $B^4$ branched along a properly embedded knotted surface. 
\end{remark} 

\begin{remark} \label{more_general}
Although Section \ref{section_symmetries} was written mainly to discuss the symmetries of 4-manifolds of width less than $\{5\}$ (see Section \ref{width_less_than_5}), one can talk about constructions of more general 4-manifolds than those that appear in Definition \ref{nerve}. Take a closed orientable surface $\Sigma$; pick finitely many points $\{t_i\}_{i=1}^{N}\subset \partial D^2$ and meridian systems $\{\alpha_i\}_{i=1}^{N}$ for handlebodies $\{H_i\}_{i=1}^{N}$. 
Let $X$ be the 4-manifold obtained by attaching 2-handles to $\Sigma\times D^2$ along $\alpha_i \times \{t_i\}$ for $i=1,\dots, N$ and capping off with 3- and 4-handles (if desired) some boundary components corresponding to Heegaard splittings of connected sums of copies of $S^2\times S^1$. 
The tuple $(\Sigma; \{H_i\}_i)$ will be the nerve of $X$ and Proposition \ref{branching_nerve} will immediately extend to this context using the same proof. One can naively ask if every 4-manifold with disconnected boundary admits such decomposition. 

\begin{prop}
Let $\Sigma$ and let $T$ be as in Remark \ref{more_general}. If $T$ is $p$-symmetric, then $X$ is the $p$-fold cover of $S^4$ with $|\partial X|$ 4-balls removed branched along a properly embedded surface.
\end{prop}
\end{remark}

\subsubsection{Drawing the branched set}
We will briefly describe how to obtain band diagrams for the branched set of a given $p$-symmetric nerve.
Recall the notation of Proposition \ref{branching_nerve} (see Fig. \ref{Mod_involution} for simplicity). In Lemma 3.3 of \cite{bridge_trisection_S4}, J. Meier and A. Zupan described an algorithm to obtain a banded link diagram (also called movie presentation) for $K^2\subset S^3\times I$ from the tuple $(S^2; \theta_\alpha, \theta_\beta, \theta_\gamma)$, which the authors called a \textbf{bridge trisection diagram} for $K^2\subset S^3\times I$. 
Using our current notation, the algorithm goes as follows: 

\begin{lem}[Lemma 3.3 of \cite{bridge_trisection_S4} rephrased] \label{bridge_to_band}
Let $(S^2;B^3,B^3,B^3)$ be a nerve for $S^3\times I$ and let $(\theta_\alpha, \theta_\beta, \theta_\gamma)$ be three collections of $b$ trivial arcs on $B^3$ with the same set of end points in $S^2$ such that $\theta_\gamma \cup \theta_\alpha$ is an unlink in $S^3=B_\alpha^3\cup B_\gamma^3$. Let $\theta_\alpha^* \subset S^2$ be a collection of shadows of $\theta_\alpha$, and excise one arc of $\theta_\alpha^*$ corresponding to each component of $\theta_\alpha\cup \theta_\gamma$. Denote by $v\subset \theta_\alpha^*$ the remaining shadows and let $C=\theta_\beta \cup \theta_\gamma \subset S^3$. Then $(C,v)$ is a banded diagram for $K^2$ in $S^3\times I$. 
\end{lem}

The bands of $v$ correspond to saddles of $K^2$ ocurring inside the collection of disks $\mathcal{D}$ that $\theta_\alpha \cup \theta_\gamma$ bounds in the 4-ball. One can change, if desired, the roles of $\gamma$ and $\alpha$ in the statement of Lemma \ref{bridge_to_band} to get the ``reversed" banded diagram for $K^2$. 

\begin{examp} 
Figure \ref{poincare_manifold} exemplifies the procedure of Proposition \ref{branching_nerve} to obtain a movie presentation for the branching surface $K^2$ in the case that $X$ admits a symmetric nerve. This breaks in four steps as follows: 
\begin{enumerate}
\item Find a system of tunnels $t$ for $L$ in $Y$ so that the corresponding nerve is $p$-symmetric. 
\item Draw three collections of $g(\Sigma)+1$ non-separating curves in $\Sigma=\partial\eta(L\cup t)$ determining the handlebodies $H=\eta(L\cup t)$, $H'=\overline{Y-H}$, $H''=H[L]$. 
\item Project the curves with the finite order map $\tau : \Sigma \ra \Sigma$ to obtain the tuple $(\theta_\alpha, \theta_\beta, \theta_\gamma)$. 
\item Apply the algorithm described in Lemma \ref{bridge_to_band} to draw the banded diagram for $K^2$.
\end{enumerate}
\begin{figure}[h]
\centering
\includegraphics[scale=.11]{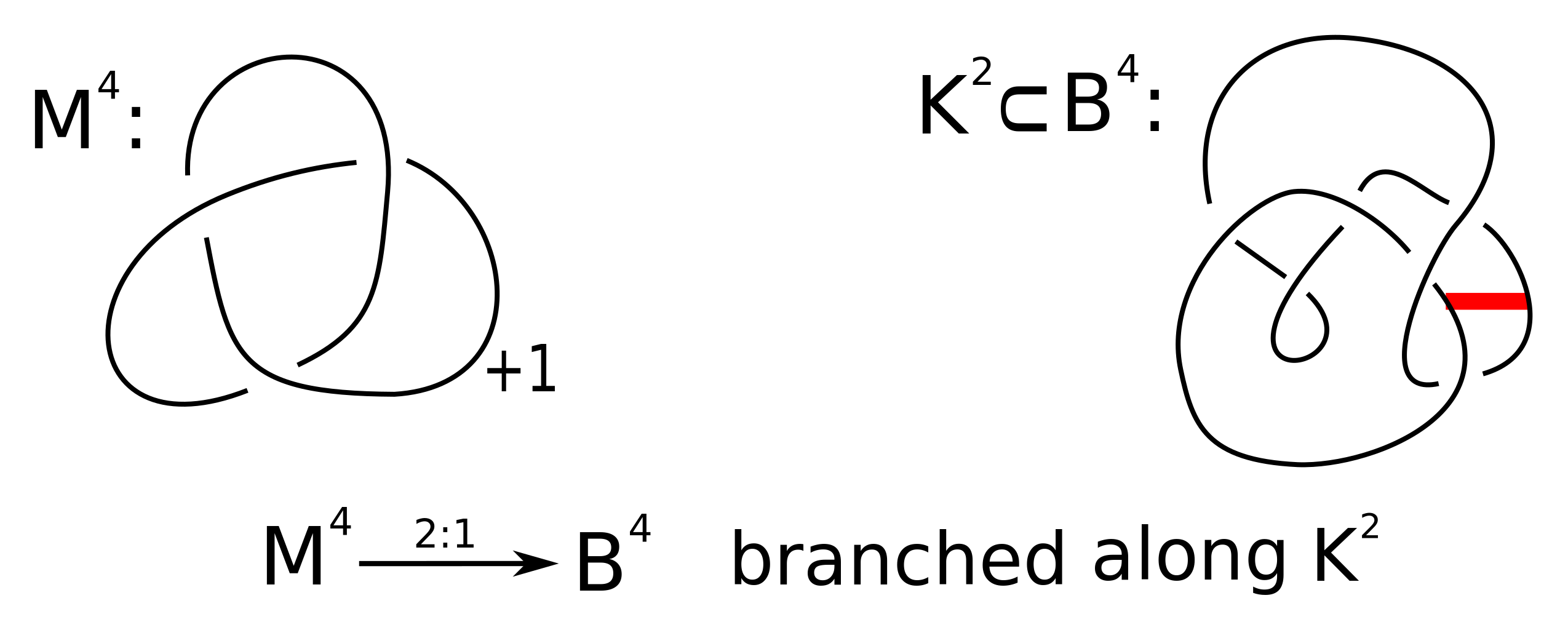}
\caption{The Poincare 4-manifold is a 2-fold cover of $B^4$ branched along a properly embedded annulus with boundary a hopf link.}
\end{figure}

\begin{figure}[h]
\centering
\includegraphics[scale=.08]{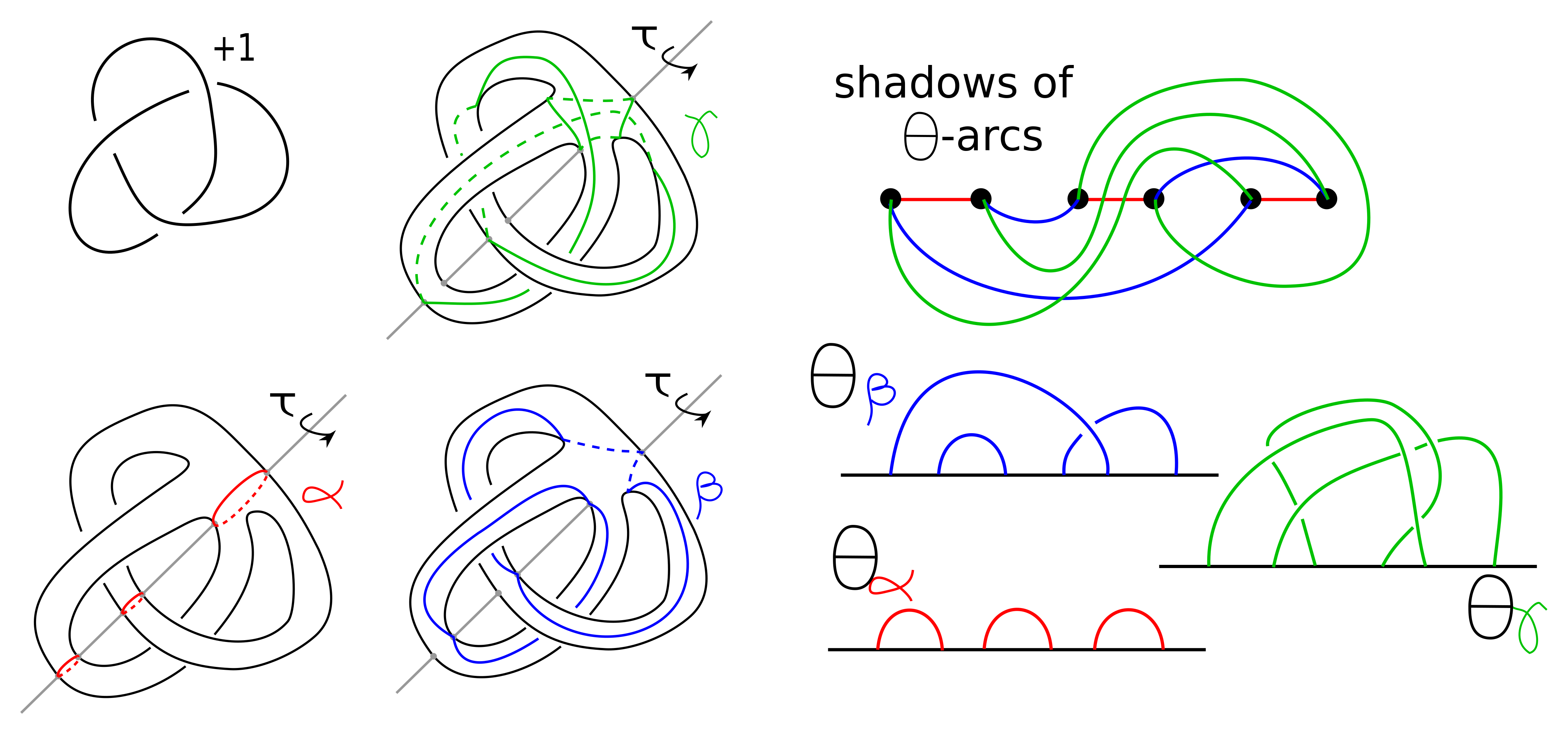}
\caption{Example of how to obtain the branched surface.}
\label{poincare_manifold}
\end{figure}
\end{examp}

\subsubsection{An application to trisections of 4-manifolds}\label{app_trisections}

We finish this subsection with an application of Proposition \ref{branching_nerve} to the theory of trisections of closed 4-manifolds. We will show that 4-manifolds with symmetric trisections are 2-fold covers of $S^4$ branched along knotted surfaces. It is important to mention that subsections \ref{extending_involutions} and \ref{width_less_than_5} are independent of the following. 

In \cite{hs_branched_covs}, J. Birman and H. Hilden studied the relations between ``p-symmetric" Heegaard splittings and branched covering representations of closed 3-manifolds. For example, they showed\footnote{This is Theorem 1 of \cite{hs_branched_covs}.} that every genus $g\geq 3$ Heegaard splitting of a closed 3-manifold may be represented as a $(4g-4)$-sheeted branched covering of $S^3$, with branching set a 1-manifold of at most $4g-4$ components. 
\begin{question}\label{question_cover}
Is it possible to extend the results of J. Birman and H. Hilden in \cite{hs_branched_covs} to the context of $p$-symmetric nerves of relative handlebodies? If so, can we do the converse of such results following M. Mulazanni's ideas in \cite{p-symmetric}? 
\end{question}

We can partially answer Question \ref{question_cover} in Theorem \ref{cor_sym_trisections}. The interested reader can compare this theorem in the case of $p=2$ with Corollary 11.3 of \cite{Akbulut_notes} or with Theorem 3 of \cite{Montesinos_4m}. One can think of this result as a trisection analogue of Theorems 2-5 in \cite{hs_branched_covs}.

\begin{thm} \label{cor_sym_trisections}
Let $M$ be a closed 4-manifold. Then $M$ is the $p$-fold cyclic cover of $S^4$ branched along a knotted surface $K^2\subset S^4$ if and only if $M$ admits a $p$-symmetric trisection diagram. 
\end{thm} 
\begin{proof} 
The forward direction was discussed in Section 2.6 of \cite{bridge_trisection_S4}.
For the backwards direction, recall that in a trisection the triplet $(\Sigma; H,H',H'')$ has the property that the 3-manifolds $H\cup_\Sigma H'$, $H'\cup_\Sigma H''$ and $H\cup_\Sigma H''$ are homeomorphic to connected sums of copies of $S^1\times S^2$. Thus for each pair, we can use the argument in Proposition \ref{branching_nerve} to extend the involution to all $M$. 
\end{proof}

\subsection{Extending involutions} \label{extending_involutions}

Let $\mathcal{H}=b_0\cup C_1\cup D_1 \cup E_1\cup \dots \cup E_N \cup b_4$ be a handle decomposition decomposition of a smooth $4-$manifold $M$. Denote by $N_i = b_0[C_1\cup D_1 \cup \dots \cup C_i]$, $M_i = N_i[D_i]$, $Y_i=\partial N_i$, $Z_i=\partial M_i$ and $X_i = \overline{M_i-N_i}$. $X_i$ is obtained by attaching 2-handles to $Y_i\times I$ along $L_i\times \{1\}$, so let $T_i=(\Sigma_i; H_i, H_i',H_i'')$ be a nerve for $X_i$ with $g(\Sigma_i)=g_i$. By construction, $\partial X_i$ is divided in two parts: $Y_i$ and $Z_i$. 
Suppose each $T_i$ is 2-symmetric. By Proposition \ref{branching_nerve}, there are involutions $\tau_i:X_i \ra X_i$ with fixed set a properly embedded surface $K_i\subset X_i$. By construction, $\tau_i|_{Y_i}$ and $\tau_i|_{Z_i}$ are induced by involutions on the handlebodies $H_i,H_i', H_i''$.

In $M$, $Z_{i-1}$ and $Y_i$ cobound a submanifold given by adding 3-handles $E_{i-1}$ and 1-handles $C_i$ to $Z_{i-1}\times I$ along $Z_{i-1}\times \{1\}$. 
\textbf{We are interested in knowing under what conditions} the involution on the pair $(Z_{i-1},Y_i)$ can be extended to the interior of such cobordism, obtaining involutions on bigger pieces of $M$. 
With the setting just described, we can extend the involution when adding 1-handles. 

%
%

\begin{lem} \label{surviving_1h}
Let $\tau$ be an involution of a closed 3-manifold $Y$ (possibly disconnected) with 1-dimensional fixed set and let $W$ be the 4-manifolds obtained from $[0,1]\times Y$ by adding a 1-handle along $\{1\}\times I$. Then there is an involution $\wt \tau$ of $W$ so that $\wt \tau|_{\{0\}\times Y}=\tau$. Furthermore, $W/\wt \tau$ is diffeomorphic to $[0,1]\times (Y/\tau)$ with a 1-handle attached. 
\end{lem} 

\begin{proof} 
Write $B^4 \approx [-1,1]\times B^3$ where $B^3\subset \C\times \R$ has coordinates $(w,t)$. Consider $g:B^4 \ra [-1,1]$ given by $g(x,w,t)=-x^2 +|w|^2 + t^2$. The map $g$ models a 4-dimensional 1-handle attachement. Define $\sigma: B^4 \ra B^4$ by $\sigma(x,w,t)=(x,-w,t)$. $\sigma$ is an involution of $B^4$ satisfying $\sigma(g^{-1}(\delta))=g^{-1}(\delta)$ for all $\delta$. 
Notice that $B^4/\sigma \approx [-1,1] \times B^3$ is again a 4-ball, and the quotient map $B^4 \ra B^4/\sigma$ is a 2-fold cover of $B^4$ branched along the 2-disk 
\[ F = Fix(\sigma) = \left\{ (x,(0,0),t): |x|\leq 1, |t|\leq 1\right\}.\]
$g|_F$ is given  by $(x,(0,0),t) \mapsto -x^2 +t^2$ which models a 2-dimensional 1-handle attachement. 
Take $V_-,V_+ \subset Y^3$ to be two disjoint closed 3-balls so that $V_\pm \cap Fix(\tau)$ is one arc. Pick coordinates for $V_\pm$ so that $V_\pm =\{ (v,t) \in \C\times \R : |v|\leq 1, |t|\leq 1\}$ and $\tau|_{V_\pm}:(v,t) \mapsto (-v,t)$. 
By hypothesis, $W\approx [0,1]\times Y^3 \cup_f B^4$ where 
\begin{align*} 
f:\{\pm1\}\times B^3 \ra &V_+\cup V_- \subset \{1\}\times Y^3\\
(+1,w,t) \mapsto & (w,t)\in V_+ \\ 
(-1,w,t) \mapsto & (w,t) \in V_-. 
\end{align*} 
By construction, $\sigma$ and $id_{[0,1]}\times \tau$ agree on a neigborhood of $V_+\cup V_-\subset \{1\}\times Y$. We then obtain the desired involution $\wt \tau=\sigma\cup_f \tau$ in $W$.
\end{proof} 

\begin{remark}\label{comments_surviving_1h}
Let $C\subset \partial W - \big(\{0\}\times Y\big)$ be the new boundary component of $W$; we have $C \approx Y \# S^1 \times S^2$. Furthermore, one can check that 
\[ Fix(\wt \tau|_{C}) = \big(Fix(\tau|_{Y^3})-int(V_+\cup V_-)\big) \cup \left\{(x,(0,0),\pm 1) : |x|\leq 1\right\} \]
Thus the action of $\wt \tau$ before and after the 1-handle attachment in Lemma \ref{surviving_1h} is described by Figure \ref{survive_1h}. 
In particular, if $V_+$ and $V_-$ are neigborhoods of points of the intersection $Fix(\tau)\pitchfork \Sigma$ where $\Sigma\subset Y$ is a Heegaard surface fixed setwise by $\tau$, then the involution $\wt \tau|_{C}$ will delete the corresponding intersection points between $\wt \tau_C$ and the new Heegaard surface. 

\begin{figure}[h]
\centering
\includegraphics[scale=.055]{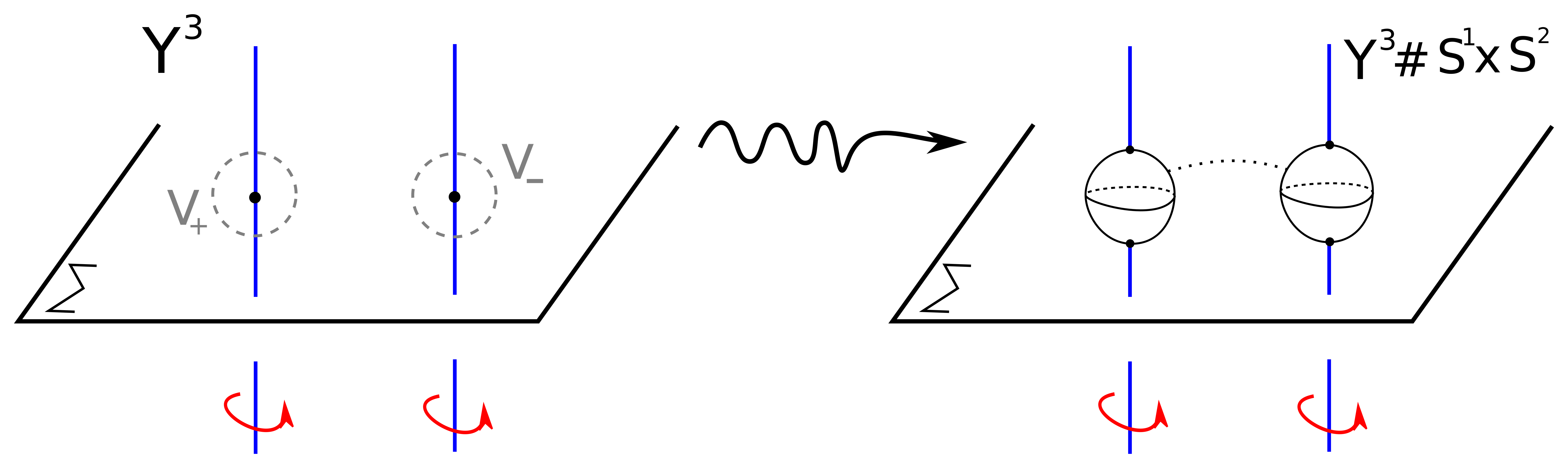}
\caption{How the extension of the involution in $Y^3$ looks once we decided the attaching region of the 1-handle. The Heegaard surface (plane in the picture) can ``survive" the 1-handle attachement.}
\label{survive_1h}
\end{figure}
\end{remark}

%

\subsubsection{When the width is less than \{5\}}\label{width_less_than_5}

In \cite{thin_position_of_3m}, M. Scharlemann and A. Thompson showed that if the width of a 3-manifold is less than \{5\}, then it is a 2-fold branched cover of a connected sum of copies of $S^1\times S^2$. In the following subsection, we will apply the previous discussion to describe an attempt of proving the analogue of this result in dimension four, say: if $M$ is a 4-manifold with $width(M)<\{5\}$, then $M$ is the 2-fold cover of connected sum of copies of $S^1\times S^3$ with $|\partial M|$ 4-balls removed, branched along a properly embedded surface.

The outline of the solution is analogous to the original one in \cite{thin_position_of_3m}: we first show that the $X_i$ blocks are branched covers of copies of $S^3\times I$ and then we study how to paste them together preserving the branched covering map. In dimension three this pasting problem resulted being a known mapping class group problem. 
From the 3-manifold theory perspective, it is interesting how this attempt reduces a 4-dimensional problem to the study of non-uniqueness of genus two Heegaard splittings of closed 3-manifolds. 

\begin{prop} \label{branching_low_width}
If $width(D)<\{5\}$, each $X_i$ admits a smooth involution $\tau_i:X_i\ra X_i$ such that the projection map $p_i:X_i \ra X_i/\tau_i \approx S^3\times I$, is a 2-fold cover of $S^3\times I$ branched along a properly embedded surface $K_i$. Furthermore, for $i=1,n$, we can take the involution to be defined on $M_1$ and $\overline{M-N_n}$, respectively.\end{prop} 
\begin{proof}
Recall that $X_i$ is obtained by attaching 2-handles to $Y_i \times I$ along $L_i \times \{1\}$. Since $\{c_i=2g_i-1\}=width(D)<\{5\}$, $g_i\leq 2$ which gives us the existance of a Heegaard splitting $H_i\cup_{\Sigma_i}H_i'$ of $Y_i$ of genus at most 2 such that $L_i\subset core(H_i)$. Then, $T_i=(\Sigma_i; H_i, H_i', H_i''=H_i[L_i])$ is a nerve for $X_i$ by Lemma \ref{lem_nerve_of_X(Y,L)}. The hyperelliptic involution on $\Sigma$ fixes the isotopy class of unoriented curves in $\Sigma$, so it extends to the interior of $H_i$, $H_i'$ and $H_i''$; thus $T_i$ is 2-symmetric for every $i$. The result follows from Proposition \ref{branching_nerve}.

When $i=1$, $H_1\cup _{\Sigma_1}H_1'$ is a Heegaard splitting for $\#_{|C_1|}S^1\times S^2$. By uniqueness of involutions on such manifolds, we will obtain that $\theta_\alpha\cup\theta_\beta$ is also an unlink in $\#_{|C_1|}S^1\times S^2$, thus we can cap $X_1$ off with 3- and 4-handles on that side and extend the surface (with its involution) with the unique boundary parallel disks; obtaining a branched cover $M_i\ra B^4$. We proceed analugously for $i=n$. 
\end{proof}

We now proceed to describe how to paste the $X_i$ blocks preserving their involutions. 
Since $width(\mathcal{H})<\{5\}$, we have that $HG(Y')\leq 2$ and $HG(Z')\leq 2$ for every connected component $Y'\subset Y_i$ and $Z'\subset Z_i$. Let $V_i=Z_i[E_i]$. For simplicity of the argument, assume $Z_i$ and $Y_{i+1}$ are connected.

Suppose first that there is a 3-handle of $E_i$ with attaching region a non-separating sphere in $Z_i$. Then $Z_i=S^1\times S^2\#L(p,q)$, with $L(p,q)$ a (possibly trivial) lens space. It follows that $Z_i$ has a unique heegaard spliting and so a unique involution. $V_i$ is a disjoint union of copies of $S^3$ and one $L(p,q)$. Take the standard involutions on then and apply the extension in Lemma \ref{surviving_1h} with the 1-handles $E_i^{op}$ to obtain an involution on the cobordism between $Z_i$ and $V_i$. Then apply Lemma \ref{surviving_1h} again to the 1-handles of $C_{i+1}$ to obtain an involution connecting $Z_i$ and $Y_{i+1}$. In this case, $Y_{i+1}$ is of the form $S^1\times S^2\#L(p,q)$ which has a unique involution. Hence, we can paste $X_i$ and $X_{i+1}$. 

Suppose now that $Z_i$ does not contain a $S^1\times S^2$ summand. Then the attaching regions of the 3-handles of $E_i$ induce a connected sum decomposition of $Z_i=\#_j X_{i,j}$ (with possibly trivial pieces). Proposition 2 of \cite{Plotnick} states that if $Z_i$ is the 2-fold cover of $S^3$ branched along $J_i$, then $J_i=\#_j J_{i,j}$ splits as a connected sum where $X_{i,j}$ is the 2-fold cover of $S^3$ along $J_{i,j}$. Taking then the involutions on each $X_{i,j}$ induced by the cover gives us an involution on $V_i=Z_i[E_i]=\coprod_{j} X_{i,i}$. By Lemma \ref{surviving_1h} to the 1-handles $E_i^{op}$ we get an involution on the cobordism between $Z_i$ and $V_i$. Applying Lemma \ref{surviving_1h} again with the 1-handles of $C_{i+1}$ gives us a involution connecting $Z_i$ with $Y_{i+1}$. 
\textbf{The issue here is} that the resulting involution in $Y_{i+1}$ may not be isotopic to the involution on $Y_{i+1}$ induced by the 2-symmety on $X_{i+1}$. 
This is due to the non-uniqueness of genus 2 Heegaard splittings for certain irreducible 3-manifolds. This motivates more the study of the problem of finding a suitable set of Montesinos moves for 2-fold branched coverings of $S^3$. 

\begin{question}
Is the space of involutions for a irreducible 3-manifold $Y$ of Heegaard genus 2 connected under a suitable topology?  
Given $\tau_0,\tau_1$ two involutions of $Y$, is there a smooth one parameter family of maps $\{\tau_s:s\in [0,1]\}$ connecting $\tau_0$ with $\tau_1$ so that for all but finitely many values $\tau_s$ is an smooth involution?
\end{question}


\end{document}